\documentclass[12pt]{article}
 \usepackage{amsmath}
\usepackage{amsfonts}
\usepackage{tikz}
\usepackage[title]{appendix}
\usepackage{subfig}
\usepackage{array}
\usepackage{blkarray}
\usepackage{nth}
\usepackage[
  separate-uncertainty = true,
  multi-part-units = repeat
]{siunitx}

\usetikzlibrary{arrows}
\usepackage{epsfig,}
\usepackage{epsfig}
\tikzset{
    vertex/.style = {
        circle,
        draw,
        outer sep = 3pt,
        inner sep = 3pt,
    },edge/.style = {->,> = latex'}
}

\usepackage{amssymb}
\usepackage{enumitem}
\usepackage{amsthm}
\usepackage{dsfont}
\def\diag{\mathop{\rm diag}}

\def\rank{\mathop{\rm rank}}
\newcommand{\rr}{\mathbb{R}}
\newcommand{\1}{\mathbf{1}}
\newcommand{\x}{\widehat{x}}
\newcommand{\q}{\widehat{q}}
\newcommand{\re}{\operatorname{Re}}

\newcommand{\0}{\mathbf{0}}

\newcommand{\la}{\lambda}

\def\tr{{\rm trace}}
\def\det{{\rm det}}

\def\cir{{\rm Circ}}

\def\L{\widetilde{L}}
\def\d{\widetilde{d}}
\def\D{\widetilde{D}}

\newtheorem{theorem}{Theorem}

\newtheorem{lemma}{Lemma}

\newcommand{\al}{\alpha}
\begin{document}
\begin{center}
\large{
Inverse formula for distance matrices of gear graphs \\}
R. Balaji and Vinayak Gupta 
\end{center}
\begin{center}
\today
\end{center}
\begin{abstract}
Distance matrices of some star like graphs are investigated in \cite{JAK}.
These graphs are trees which are stars, wheel graphs, helm graphs and
gear graphs. 
Except for gear graphs in the above list of star like graphs, there are precise formulas available in the literature to compute the inverse/Moore-Penrose inverse of their
distance matrices. These formulas tell that if $D$ is the distance matrix of $G$, then
\[D^{\dag}=-\frac{1}{2}L+uu', \] 
where $L$ is a Laplacian-like matrix which is positive semidefinite and
all row sums equal to zero. The matrix $L$ and the vector $u$ depend
only on the degree and number of vertices in $G$ and hence, can be written directly from $G$. The earliest formula obtained is for distance matrices of trees 
in Graham and Lov\'{a}sz \cite{GL}.
In this paper, we obtain an elegant formula of this kind to compute the 
Moore-Penrose inverse of the distance matrix of a gear graph.
\end{abstract}

{\bf AMS CLASSIFICATION.} 05C50

\section{Introduction}
All graphs in this paper will be simple.
Let $G$ be a connected graph with $m$ vertices labelled $\{1,\dotsc,m\}$.
The distance between any two vertices $u$ and $v$, denoted by $d_{uv}$, is the
length of the shortest path in $G$ from $u$ to $v$. 
The distance matrix of $G$, denoted by $D(G)$, is then the
$m \times m$ matrix with $(i,j)^{\rm th}$ entry equal to 
\[
\begin{cases}
0 &  {i= j}\\
d_{ij} & {\mbox{else}} .   \\
\end{cases}
\]
The theory of distance matrices of graphs begins with an interesting result of Graham and Pollak \cite{GRA} saying that
if $G$ is a tree on $m$ vertices, then  
\[\det(D(G))=(-1)^{m-1}(m-1)2^{m-2}.\] 
Hence, if $G_1$ and $G_2$ are any two trees with same number of vertices, then
$\det(D(G_1))$ and $\det(D(G_2))$ are equal. Motivated by an application in a data communication problem, 
Graham and Lov\'{a}sz \cite{GL} investigated distance matrices of trees 
more extensively where the following remarkable inverse formula is obtained.
Let $T$ be a tree on $m$ vertices and let $D:=D(T)$. Then, 
\begin{equation} \label{grl}
D^{-1}=-\frac{1}{2}S+\frac{1}{2(n-1)} \tau \tau',
\end{equation}
where $\tau=(2-\delta_1,\dotsc,2-\delta_{m})'$ and $S$ is the Laplacian matrix of $T$.
(Here $\delta_i$ is the degree of vertex $i$.) By the multilinearlity property of the
determinant, the formula for $\det(D(G))$ can be deduced from (\ref{grl}). Recall that the Laplacian matrix of a connected graph is
$\Delta-A$, where $\Delta:=\diag(\delta_1,\dotsc,\delta_m)$ and $A$ is the adjacency matrix. Thus, the inverse of $D$ just becomes a direct computation from  the degree sequence of the tree.
Formula (\ref{grl}) has an extension to weighted trees; \cite{KIRK}. Suppose to each edge 
$(i,j)$ of $T$, a positive number $w_{ij}$ is assigned. 
The scalars $w_{ij}$ will be called the weights of $T$. 
The distance $\d_{ij}$ in this set up
is the weighted distance which is the sum of all the weights in the path
connecting $i$ and $j$. Define $\D:=[\d_{ij}]$. 
The formula in \cite{KIRK} says that
\begin{equation} \label{grlw}
{\D}^{-1}=-\frac{1}{2}\widetilde{L}+ \frac{1}{2\sum \limits_{i,j}w_{ij}} \tau \tau',
\end{equation}
where the weighted Laplacian $\widetilde{L}=[\xi_{ij}]$ is defined as follows:
\[\xi_{ij}:=
\begin{cases}
\sum \limits_{i  \sim  k} w_{ik}^{-1} &  {i= j}\\
-w_{ij}^{-1} & {i \sim  j}    \\
0   &           \mbox{else}.
\end{cases}
\]
Clearly, $\widetilde{L}$ depends only on the weights. Again, the vector $\tau$ is given by (\ref{grl}).
Specializing $w_{ij}=1$ for all $i$ and $j$, formula (\ref{grlw}) reduces to (\ref{grl}).
In both the weighted and unweighted cases, the following identity that connects the Laplacian and
distance matrices plays a crucial role in deducing $(\ref{grl})$ and $(\ref{grlw})$:
\[d_{ij}=l^{\dag}_{ii} + l^{\dag}_{jj}-2 l^\dag_{ij}~~\mbox{and}~~\d_{ij}=\xi_{ii}^{\dag}+\xi_{jj}^{\dag}-2 \xi_{ij}^{\dag}, \]
where $l^{\dag}_{ij}$ and $\xi_{ij}^{\dag}$ are the $(i,j)^{\rm th}$-entries of the Moore-Penrose inverse of $L$ and $\L$ respectively.
Both $L$ and $\L$ are positive semidefinite matrices. This is easy to see. 
So, distance matrices of trees are Euclidean distance matrices.
We say that an $m \times m $ matrix $D=(d_{ij})$ is called a Euclidean distance matrix (EDM) if there exist $p_1, \dotsc, p_m$ in some Euclidean space such that 
\[ d_{ij}=\|p^i-p^j\|^2  ~~i,j=1,\dotsc,m.    \]   
Equivalently, $D$ is a EDM if and only if there exists a positive semidefinite
matrix $F=[f_{ij}]$ such that
\[d_{ij}=f_{ii}+f_{jj}-2f_{ij}. \]

A wide literature on EDMs appear in \cite{ALF}. To compute the Moore-Penrose inverse 
of a EDM, the following formula is deduced in Balaji and Bapat \cite{RB}: If 
$D$ is a EDM and 
$\1'D^{\dag}\1>0$, then
\[D^{\dag}=-\frac{1}{2}U^\dagger+\frac{1}{\1'D^\dag \1  }(D^\dag \1)(D^\dag \1)', \]
where $U:=-\frac{1}{2}PDP$, $P:=I-\frac{J}{m}$, $J:=\1 \1'$ and $\1$ is the vector of all ones in $\rr^m$.
Using this formula, it is easy to get (\ref{grl}) and $(\ref{grlw})$, as distance matrices of trees are EDMs.

A general question is:
Given a connected graph, deduce a formula 
to compute the inverse/Moore-Penrose inverse of its distance matrix.
Unlike distance matrices of trees, 
there are connected graphs whose distance matrices are not  EDMs.
In the recent times, there are some interesting results on connected graphs whose distance matrices are EDMs. 
Suppose $G$ is a connected graph and $D(G)$ is an EDM.
To get an inverse formula for $D(G)$, a special Laplacian needs to be constructed. 
Usually, based on the observations of several numerical computations, this matrix is identified. Then a rank one vector is identified. Thus, as in the right hand side of (\ref{grl}), a matrix  is constructed. Then it is shown that this is the inverse of $D(G)$ by some special techniques.
 Our work in this paper 
follow on these lines.

Jakli\'{c} and Modic \cite{JAK} identified some
star-like graphs which are EDMs. To be precise, a star-like graph is a connected graph
$G$ that has a subgraph $H$ which is a star tree.  
Star-like graphs considered in \cite{JAK} are wheel graphs, helm graphs
and gear graphs.
Spectral properties of the distance matrices of these graphs were obtained 
 in \cite{JAK}.
 The inverse formula in the spirit of Graham and Lov\'{a}sz
for wheel graphs and helm graphs are deduced in \cite{GWE}, \cite{GW} and \cite{GH}.    
In this paper, we derive a formula for the Moore-Penrose inverse of the
distance matrix of a gear graph.
A gear graph is constructed from a wheel graph. Let $W_n$ be a wheel graph  
with $n$ vertices. Then the gear graph is obtained by inserting a new vertex in between any two vertices of the outer cycle of $W_n$. Hence the number of vertices in this gear graph is $2n-1$. So, the distance matrix of this gear graph has order $2n-1$.
Furthermore, as pointed out in Section $5$ of \cite{JAK}, the distance matrix
of a gear graph with $2n-1$ vertices has the form
\begin{equation} \label{ddefg}
\left[
\begin{array}{ccccccc}
0 & \1_{n-1}' & 2 \1_{n-1}' \\
\1_{n-1} & 2 (J_{n-1}-I_{n-1}) & S\\
2 \1_{n-1} & S' & T \\  
\end{array}
\right],
\end{equation}
where $S$ is not symmetric. The structure of the distance matrix of wheel and helm graphs are simpler. In fact, all the square blocks in the distance matrices of wheel and helm graphs are symmetric. Unlike this, we have a square block $S$ in (\ref{ddefg}) which is not symmetric.
In addition, distance matrix of a gear graph of order $2n-1$ has a much lower rank.
The rank is $n$. Due to these facts,   
to get a formula for the Moore-Penrose inverse in the gear case is more complicated than in the earlier cases.
In this paper, we accomplish this task.

\section{Preliminaries}
The following notation will be used in this paper.
\subsection{Notation}
\begin{enumerate}
\item[\rm (i)] All vectors are assumed column vectors. To denote the transpose of an $m \times n$ matrix $A$,
we use $A'$ and for the conjugate transpose, we use $A^*$. As usual, the Moore Penrose inverse of $A$ will be written $A^\dag$. The real part of $A$ will be denoted by
$\re{A}$. The notation $A(r,s)$ will denote the $(r,s)^{\rm th}$-entry of $A$.
 If $v$ is a vector with $m$ components, then $v(r)$ will denote the $r^{\rm th}$ coordinate of $v$.

\item[\rm (ii)] Throughout the paper, we assume that 
$n$ is a positive integer and $n \geq 4$.
The notation $\omega$ will be used for the $(n-1)^{\rm th}$ root of unity:  
 \[\omega:=e^{\frac{2\pi i}{n-1}}.\]
So, $1,\omega,\dotsc,\omega^{n-2}$ are the roots $z^{n-1}=1$.
\item[\rm (iii)] 
The notation $\1_{m}$ will denote the vector of all ones in $\rr^m$.
For the $m \times m$ matrix of all ones, we shall use $J_{m}$.
The notation $I_{m}$ will stand for the
$m \times m$ identity matrix. 
 If $m=2n-1$, we shall simply use
 $\1$, $J$, and $I$.  
We use $\0$ for a column/row vector where all entries are zero. To denote the zero matrix with more than one row/column, we shall use $O$.

\item[\rm (iv)] The inner product between the vectors $x$ and $y$ in a Euclidean/unitary space
will be written 
$\langle x,y \rangle$. 
 
\item[\rm (v)] The $n \times n$ circulant matrix specified by the row vector 
$(a_1,\dotsc,a_n)$ will be denoted by $\cir(a_1,\dotsc,a_n)$. We recall that
\begin{equation*}
\cir(a_1,\dotsc,a_n)= \left[
{\begin{array}{rrrrrr}
a_1 & a_2 & a_3 & \ldots & a_n \\
a_n & a_1 & a_2 & \ldots & a_{n-1} \\
a_{n-1} & a_n & a_1 &  \ldots &a_{n-2} \\
\vdots & \vdots & \vdots &  \ddots &\vdots \\
a_2 & a_3 & a_4 &  \ldots &a_1
\end{array}}
\right].
\end{equation*}

\item[\rm (vi)] Let $W_n$ be the wheel graph with $n$ vertices. 
To each edge in the outer cycle of $W_n$, insert a new vertex. The resulting graph 
is called the gear graph obtained from $W_n$. 
We denote this by $G_n$. 
The number of vertices in $G_n$ is $2n-1$. 
\end{enumerate}

  \begin{figure}
\centering 
\begin{tikzpicture}  
  [scale=1.3,auto=center,every node/.style={circle,fill=black!20}] 
    
  \node (a1) at (0,0) {$g_1$};  
  \node (a2) at (1.5,-2)  {$g_2$};  
  \node (a3) at  (2.5,0)  {$g_3$};  
  \node (a4) at  (1.5,2)  {$g_4$};  
  \node (a5) at (-1.5,2)   {$g_5$};  
  \node (a6) at  (-2.5,0) {$g_6$};  
   \node (a7) at  (-1.5,-2) {$g_7$}; 
   \node (a8) at (2,-1)  {$g_8$}; 
   \node (a9) at (2,1)  {$g_9$}; 
    \node (a10) at (0,2)  {$g_{10}$};
   \node (a11) at (-2,1)  {$g_{11}$}; 
    \node (a12) at (-2,-1)  {$g_{12}$};
    \node (a13) at (0,-2)  {$g_{13}$};

  \draw (a1) -- (a2); 
  \draw (a2) -- (a8);  
  \draw (a2) -- (a13);  
  \draw (a4) -- (a9);  
  \draw (a6) -- (a12);  
  \draw (a3) -- (a8);  
   \draw (a1) -- (a5);
     \draw (a3) -- (a1); 
     \draw (a1) -- (a6);   
      \draw (a4) -- (a10); 
         \draw (a7) -- (a1); 
          \draw (a7) -- (a12); 
          
        \draw (a1) -- (a4);
     \draw (a3) -- (a9); 
     \draw (a11) -- (a6);   
      \draw (a5) -- (a10); 
         \draw (a7) -- (a13); 
          \draw (a5) -- (a11);   
          
\end{tikzpicture}  
 \caption{$G_{13}$ obtained from $W_7$} \label{fig_G13}
\end{figure}
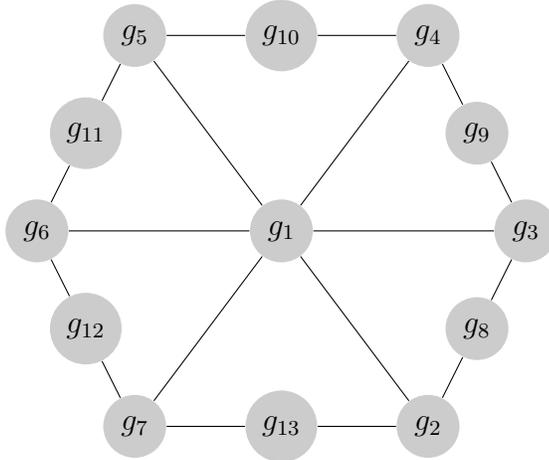

\subsection{Distance matrix of a gear graph}
We fix a labelling for $G_n$. The vertices of the subgraph 
$W_n$ are first labelled as follows.
The center of $W_n$ is labelled $g_1$ and 
the outer cycle of $W_n$ are labelled $g_2,g_3,\dotsc,g_n$ anticlockwise. 
Now, label the vertex in between the edge $(g_i,g_{i+1})$ of $W_n$ by $g_{i+n-1}$.
We now have a labelling of all the $2n-1$ vertices of $G_n$.
To illustrate, $G_{13}$ is given in Figure \ref{fig_G13}.

With this labelling, the distance matrix of the gear graph $G_n$ will be
\begin{equation} \label{dgn}
D(G_n)=\left[
\begin{array}{ccccccccc}
0 & \1_{n-1}' & 2 \1_{n-1}'\\
\1_{n-1} & 2(J_{n-1}-I_{n-1}) & S\\
2 \1_{n-1} & S' & T
\end{array}
\right],
\end{equation}
where $S:=\cir(1,\underbrace{3,\dotsc,3}_{n-3},1)$ and $T:=\cir(0,2,\underbrace {4,\dotsc,4}_{n-4},2)$.
\subsection{Known results}
We use the results that are proved in Section $5$ of \cite{JAK}
extensively. 
\begin{enumerate}
\item[\rm (K1)] The null-space of $D(G_n)$
is computed precisely in $\cite{JAK}$. 
Let $e_i \in \rr^{n-1}$ be the
vector having $1$ in the $i^{\rm th}$ position and zeros elsewhere. 
Define
\begin{equation} \label{null}
{f^{i}}:=
\begin{cases}
\large(1,-(e_{i}'+e_{i+1}'),e_{i}'\large)^{'}&  i=1,2,\dotsc,n-2\\
\large(1,~-(e_{1}'+e_{n}'), ~e_{n-1}' \large)'&  i=n-1.
\end{cases}
\end{equation}
Then,
$\{f^{1},\dotsc,f^{n-1}\}$ is a basis for the null-space of $D(G_n)$. 
We note that 
\begin{equation} \label{1fi=0}
\langle \1,f^{i}\rangle=0~~~i=1,\dotsc,n-1
\end{equation}
and $\rank(D(G_n))=n$.
 
\item[\rm (K2)]
The non-zero eigenvalues of $D(G_n)$ are
given in the proof of Theorem 13 in \cite{JAK}.  
The real numbers
\begin{equation} \label{lambdaj}
\lambda_{j}:=3n-8-(-1)^{j}\sqrt{5}\sqrt{n(2n-9)+12}~~~~j=1,2
\end{equation}
 are eigenvalues of $D(G_n)$.
Eigenvectors corresponding to $\lambda_{j}$ are given by
\begin{equation} \label{xj}
x_{j}=(\alpha_{j},\beta_{j}\1_{n-1}',\1_{n-1}' )^{'},
\end{equation}
where
\[\alpha_{j}:=\frac{15-5n-2(-1)^{j}\sqrt{5}\sqrt{n(2n-9)+12}}{3(n-1)}, \]
\[ \beta_{j}:=\frac{6-n-(-1)^{j}\sqrt{5}\sqrt{n(2n-9)+12}}{3(n-1)}.\]
The other $n-2$ non-zero eigenvalues of $D(G_n)$ are  
\begin{equation} \label{thetak} 
 \theta_{k}:=-8\cos^{2}(\frac{\pi k}{n-1})-2~~~k=1,\dotsc,n-2.
 \end{equation}
The result in Theorem $13$ of \cite{JAK} says that $D(G_n)$ is an Euclidean distance matrix (EDM).
By a well-known result, an $m \times m$ symmetric matrix $G$ is a EDM if it has zero diagonal
and \[x'Gx \leq 0~~\mbox{for all} ~~x \in \1_{m}^{\perp}.\]

\item[\rm (K3)] \label{itm:first} 
The following result is well known on circulant matrices; see \cite{REV}.
The eigenvalues of $\cir( c_0,c_1,\dotsc, c_{n-2})$ 
are 
\[\sigma_ m  = \sum_{p=0}^{n-2} c_{p}\omega^{pm}~~~m=0,\dotsc,n-2.\]
The eigenvectors corresponding to $\sigma_m$ are given by \[v_{m}=(1,\omega^{m},\omega^{2m},\dotsc ,\omega^{m(n-2)})^{'}~~m=0,\dotsc,n-2.\] Furthermore, the vectors $v_0,\dotsc,v_{n-2}$ are mutually orthogonal.

\item[\rm (K4)] The following theorem is obtained in \cite{RB}. Let $D$ be an   $m\times m$ EDM. Suppose $\beta:=\1_{m} D^\dagger \1_{m} > 0$. Let $P=I_{m}-\frac{J_{m}}{m}$, $G=-\frac{1}{2}PDP$ and $u=D^\dagger \1$. Then
\[D^\dagger =-\frac{1}{2}G^\dagger+\frac{1}{\beta}uu'.\]
\end{enumerate}
\section{Main results}
We shall consider distance matrices of $G_n$ for the cases $n$ even and $n$
odd separately. The formulas to compute the Moore-Penrose inverse are different in
both these cases. We first consider even case. 
\subsection{Moore-Penrose inverse of $D(G_{n})$ when $n$ is even}
Consider a gear graph $G_n$, where $n$ is even. Motivated by numerical computations, we define a special Laplacian matrix for $G_n$ now.
\subsection{Special even Laplacian} \label{speciallaplacianeven}
For the gear graph $G_n$ with $n$ even, the special Laplacian matrix will be denoted by $L_{even}$. 
Define \[\phi_k:= \cos(\frac{\pi k}{n-1})~~~k=1,\dotsc,n-2. \]
Since $n-1$ is odd, \[\phi_{k} \neq 0~~\mbox{for all}~k.\]
Define 
\begin{equation} \label{A}
A:=\frac{9(n-1)}{(n+4)^2}
\left[
\begin{array}{ccccccccc}
1 & \frac{(n-2)}{3(n-1)}\1_{n-1}' &\frac{-(n+1)}{3(n-1)}\1_{n-1}' \\
\frac{(n-2)}{3(n-1)}\1_{n-1}& \frac{(n-2)^2}{9(n-1)^2}J_{n-1} &\frac{-(n-2)(n+1)}{9(n-1)^2}J_{n-1} \\
\frac{-(n+1)}{3(n-1)} \1_{n-1}& \frac{-(n-2)(n+1)}{9(n-1)^2}J_{n-1} & \frac{(n+1)^2}{9(n-1)^2}J_{n-1}
\end{array}
\right],
\end{equation}

\begin{equation}\label{bkdef}
B_{k}:=  \frac{2}{(n-1)(2\phi_k+ \frac{1}{                                                        2\phi_k})^{2}} 
\left[
\begin{array}{ccc} 
 0&\0&\0\\
\0& \frac{C_k}{4\phi_{k  }^{2}}& \frac{C_k+\widetilde{C_{k}}}{ 4\phi_k^{2}  }\\ 
\0 & \frac{C_k+\widetilde{C_{k}}^{'}}{ 4\phi_{k}^{2}} & C_k \\
\end{array}
\right]~~~k=1,\dotsc,n-2,
\end{equation}
 where
 \[C_{k}=\cir(1,\cos(\frac{2\pi k }{n-1}),\dotsc,\cos(\frac{2 \pi (n-2)  k }{n-1}))\]
 \[\widetilde{C_{k}}= \cir(\cos(\frac{2 \pi k}{n-1}),\dotsc,\cos(\frac{2 \pi (n-2) k }{n-1}),1).\] 
 We now define the special Laplacian matrix for the gear graph $G_n$ by
 \begin{equation}\label{speciallap}
L_{even}:=A+ \sum_{k=1}^{\frac{n-2}{2}}B_k.
\end{equation}

\subsection{Inverse formula for $D(G_n)$}
We shall obtain the following result that gives the formula for the Moore-Penrose of $D(G_n)$.
\begin{theorem}\label{main}
Let $n \geq 4$ be an even integer. Then the Moore-Penrose inverse of 
$D(G_n)$ is given by 
\[D(G_n)^\dag = -\frac{1}{2}L_{even}+\frac{n-1}{2}uu',\]
where 
\[u= \huge(\frac{-(3n-13)}{(n^2+3n-4)},\frac{-(n-6)}{(n^2+3n-4)}\1_{n-1}',\frac{1}{(n+4)}\1_{n-1}' \huge)'.\]
\end{theorem}

\subsection{Illustration for $G_6$}
The formula in Theorem \ref{main} is easy to verify for $G_4$. The distance matrix of the gear graph 
$G_6$ is 
\[D(G_6)= \left[
\begin{array}{ccccccccc}0& \1_{5}'&2\1_{5}'\\
\1_{5}&2(J_{5}-I_{5})& \cir( 1,3,3,3,1)\\
2\1_{5}&\cir(1,1,3,3,3)&\cir(0,2,4,4,2)
 \end{array}
\right].
\]
The Laplacian $L:=L_{even}$ of $G_6$ is given by
\[L=
  \left[
\begin{array}{ccccccccc}
\frac{9}{20}& \frac{3}{25}\1_{5}'   &\frac{-21}{100}\1_{5}'\\
 \frac{3}{25}\1_{5}&\cir( \frac{34}{125}, \frac{-16}{125},   \frac{9}{125},\frac{   9}{125},\frac{ -16}{125})&\cir( \frac{3}{125},\frac{-22}{125},  \frac{3}{125},\frac{   -22}{125},\frac{ 3}{125})\\
 \frac{-21}{100}\1_{5}'&
\cir( \frac{ 3}{125},\frac{3}{125},\frac{-22}{125},  \frac{3}{125},\frac{   -22}{125})&\cir( \frac{ 129}{500},\frac{29}{500},\frac{29}{500},  \frac{29}{500},\frac{   29}{500})
\end{array}
\right].
\]
The vector $u$ in Theorem \ref{main} is
\[u=(\frac{1}{10},\0,\frac{1}{10}\1_{5}')'.\]
According to Theorem \ref{main},
\[D^\dagger =-\frac{1}{2}L +\frac{5}{2}uu'.\]

The right hand side of the above equation is
\[\widetilde{D}=
  \left[
\begin{array}{ccccccccccccccc}
\frac{-1}{5}& \frac{-3}{50}\1_{5}'   &\frac{2}{25}\1_{5}'\\
 \frac{-3}{50}\1_{5}
 &\cir( \frac{-17}{125}, \frac{8}{125},   \frac{-9}{250},\frac{   -9}{250},\frac{ 8}{125})
 &\cir( \frac{-3}{250},\frac{11}{125},  \frac{-3}{250},\frac{   11}{125},\frac{ -3}{250})\\
 
 \frac{2}{25}\1_{5}'&
 \cir(\frac{ -3}{250}, \frac{-3}{250},\frac{11}{125},  \frac{-3}{250},\frac{   11}{125})
 &\cir( \frac{ -13}{125},\frac{-1}{250},\frac{-1}{250},  \frac{-1}{250},\frac{   -1}{250})
\end{array}
\right].
\]
It can be verified that $D^\dag=\widetilde{D}$. We shall now prove the formula for any
general $n$. For brevity, we shall use $D$ for $D(G_n)$.

\subsection{Eigenvalues and eigenvectors of $D$}
To prove Theorem \ref{main}, we precisely compute all the non-zero eigenvalues and the corresponding eigenvectors of $D$. This is partly done in
(K2). The nonzero eigenvalues of $D$ are given by $\lambda_1,\lambda_2$ and 
$\theta_1,\dotsc,\theta_{n-2}$ in (K2). 
In the sequel, $j$ will denote an element in $\{1,2\}$
and $k$ in $\{1,\dotsc,n-2\}$.
Eigenvectors $x_{j}$ corresponding to 
$\lambda_{j}$ are given in $(\ref{xj})$. We now compute the eigenvectors
$q_k$ corresponding to $\theta_k$.
The following elementary lemma will be useful in the sequel.

\begin{lemma}\label{Z2}
\begin{enumerate}
\item[\rm (i)] The eigenvalues of $\cir(0,2,\underbrace{4,\dotsc,4}_{n-4},2)$ are
\[4(n-3),-8\cos^{2}(\frac{\pi}{n-1}), -8\cos^{2}(\frac{2\pi}{n-1})
\dotsc,-8\cos^{2}(\frac{(n-2)\pi}{n-1}).
\]
\item[\rm (ii)] The eigenvalues of $\cir(1,\underbrace{3,\dotsc,3}_{n-3},1)$ are
\[3n-7,-2(1+\omega^{-1}),\dotsc,-2(1 +\omega^{-(n-2)}).\] 
\end{enumerate}
\end{lemma}
\begin{proof}
Put $T:=\cir(0,2,\underbrace{4,\dotsc,4}_{n-4},2)$.
Since \[T \1_{n-1}=4(n-3) \1_{n-1},\]
$4(n-3)$ is an eigenvalue of $T$.
By (K3), 
\[\alpha_{m}:=2\omega^{m}+4\omega^{2m}+4\omega^{3m}+\dotsc +4\omega^{(n-3)m}+2\omega^{(n-2)m}~~m=1,\dotsc,n-2,\]
are eigenvalues of $T$. 
We rewrite $\al_{m}$ as
\[\al_{m}=4(\omega^{m}+\omega^{2m}+\dotsc + \omega^{(n-2)m})-2\omega^{m}-2\omega^{(n-2)m}. \]
By the identities \[\omega^{m}+\omega^{2m}+\dotsc + \omega^{(n-2)m}=\frac{(\omega^{m}-\omega^{(n-1)m})}{ (1-\omega^{m})},\]
$\omega^{(n-1)m}=1$ and $\omega^{(n-2)m}=\omega^{-m}$,
 it now follows that
 \[\al_{m}=-4-2\omega^{m}-2\omega^{-m}.\]
Because $\omega^{m}+\omega^{-m}=2\cos(\frac{2 \pi m}{n-1})$,
 \[\al_{m}=-4-4\cos(\frac{2 \pi  m}{n-1}).\]
 Since $1+\cos(\frac{2 \pi m}{n-1})=2\cos^{2}(\frac{\pi m}{n-1})$,
 \[\al_{m}=-8\cos^{2}(\frac{\pi m}{n-1}).\]
 The proof of (i) is complete.

Define $S:=\cir(1,\underbrace{3,\dotsc,3}_{n-3},1)$.
Since \[S \1_{n-1}=(3n-7) \1_{n-1},\] $3n-7$ is an eigenvalue of $S$.
By (K3),
\[\beta_m= 1+3\omega^{m}+\cdots +3\omega^{(n-3)m}+\omega^{(n-2)m}~~m=1,\dotsc,n-2\]
are eigenvalues of $S$. 
We rewrite the above equation as
\[\beta_m=3(1+\omega^{m}+\cdots+\omega^{(n-2)m})-2(1+\omega^{(n-2)m}).\]
Since $1+\omega^{m}+\cdots+\omega^{(n-2)m}=0$,
we get
\[\beta_{m}=-2(1+\omega^{(n-2)m}). \]
 By the identity $\omega^{(n-2)m}=\omega^{-m}$,
 \[ \beta_m=-2(1+\omega^{-m}).\]
This completes the proof of (ii).
\end{proof}
We now compute the eigenvectors corresponding to the eigenvalue $\theta_k$ in (K2).
\begin{theorem}\label{ee}

Define
\[ q_{k}:=(0,u_k',v_k')'~~~k=1,\dotsc,n-2,\]
where 
\[u_{k}=\frac{-1}{8\cos^{2}(\frac{\pi k}{n-1})}S v_{k},~~S=\cir(1,\underbrace{3,\dotsc,3}_{n-3},1)~~\mbox{and}\]
\[v_{k}=(1,\omega^{k},\omega^{2k},\dotsc,\omega^{k(n-2)})'.\]
Then, 
\begin{enumerate}
\item[\rm (i)] $Dq_k=\theta_k q_k$.
\item[\rm (ii)] The vectors $q_1,\dotsc,q_{n-2}$ are mutually orthogonal.
\item[\rm (iii)] $\langle x_j,  q_{k}\rangle=0$ for all $j,k$.
\item[\rm (iv)] $\langle \1, q_k\rangle=0$. 
\end{enumerate}

\end {theorem}
\begin{proof}
Fix $k \in \{1,\dotsc,n-2\}$.
Recall that 
\[D=\left[
\begin{array}{ccccccccc}
0 & \1_{n-1}' & 2 \1_{n-1}'\\
\1_{n-1} & 2(J_{n-1}-I_{n-1}) & S\\
2 \1_{n-1} & S' & T
\end{array}
\right].\]
Since $q_k=(0,u_k',v_k')'$,
\begin{equation}\label{q0}
D q_k=
\left[
\begin{array}{cccc}
\1_{n-1}'u_k+2 \1_{n-1}'v_k\\
2J_{n-1}u_k-2I_{n-1}u_k+Sv_k\\
S'u_k+Tv_k
\end{array}
\right].
\end{equation}
We claim that $\nu_k:=\1_{n-1}'u_k+2 \1_{n-1}'v_k=0$.
As
\begin{equation} \label{uk}
u_k=\frac{-1}{8\cos^{2}(\frac{\pi k}{n-1})}Sv_k,
\end{equation}
\begin{equation}\label{q1}
  \nu_k =(\frac{-1}{8\cos^{2}(\frac{\pi k}{n-1})})\1'_{n-1}S v_{k}+2 \1_{n-1}'v_k.
\end{equation}  
By Lemma $\ref{Z2}$, 
$-2(1+\omega^{-k})$ is an eigenvalue of $S$. In view of
(K3), \[S v_{k}=-2(1+\omega^{-k} )v_{k},\]
where
$v_{k}=(1,\omega^{k},\omega^{2k},\dotsc,\omega^{k(n-2)})'$. 
 Thus from equation $(\ref{q1})$, we have,
 \begin{equation}\label{q2}
\nu_k =\frac{2(1+\omega^{-k} )}{8\cos^{2}(\frac{\pi k}{n-1})}\1_{n-1}'   v_{k}+2\1_{n-1}'   v_{k}.
\end{equation} 
Since $S$ is a circulant matrix, the result in
(K3) can be applied to $S$. This tells that
$\1_{n-1}$ and $v_k$ are orthogonal.
 Thus by (\ref{q2}),
\begin{equation*}\label{q3}
\nu_k=0.
\end{equation*}
The claim is proved.

Define
\[w_{k}:=2J_{n-1}u_k-2I_{n-1}u_k+Sv_k.\]
We show that $w_k=\theta_k u_k$. 
Since $J_{n-1}=\1_{n-1} \1_{n-1}^{'}$, we have
\[w_{k}=2\1_{n-1}\1_{n-1}'u_k-2u_k+Sv_k. \]
Because $\1_{n-1}' v_k=0$ and $\nu_k=0$, we get 
\begin{equation} \label{1'uk=0}
\1_{n-1}'u_k=0.
\end{equation}
So,
$w_{k}=-2u_{k}+Sv_{k}$.
By equation (\ref{uk}), we now have
\begin{equation*} \label{2j}
w_{k}=(\frac{1}{4\cos^{2}(\frac{\pi k}{n-1})}+1)S v_{k}~~\mbox{and}~~
Sv_k=-8\cos^{2}(\frac{\pi k}{n-1})u_k.
\end{equation*}
 This leads to
\begin{equation} \label{2k}
\begin{aligned}
w_k&=(\frac{1}{4\cos^{2}(\frac{\pi k}{n-1})}+1)(-8\cos^{2}(\frac{\pi k}{n-1}))u_k\\
&=-2(1+4\cos^{2}(\frac{\pi k}{n-1}))u_k.
\end{aligned}
\end{equation}

We recall from equation (\ref{thetak}) that
\[\theta_{k}:=-8\cos^{2}(\frac{\pi k}{n-1})-2. \]
Hence from (\ref{2k}), we get
\begin{equation}\label{q5}
w_k=\theta_k u_k.
\end{equation}

Define
\[l_k:=S'u_k+Tv_k.\]
We show that $l_{k}=\theta_k v_k$. 
By (\ref{uk}),   
\begin{equation}\label{q6}
S'u_k+Tv_k=(\frac{-1}{8\cos^{2}(\frac{\pi k}{n-1})})S'Sv_k+Tv_k.
\end{equation}
By  Lemma $\ref{Z2}$ and  (K3),
 \[Sv_k=-2(1+\omega^{-k})v_k.\]  
Because $S$ is a circulant matrix, it is normal. Hence,
 \[S'v_k=-2(1+\omega^{k})v_k.\]   
Again by Lemma $\ref{Z2}$ and (K3), we have  
 \[Tv_k=-8\cos^{2}(\frac{\pi k}{n-1})v_k.\]
 Thus from $(\ref{q6})$, we have 
 \begin{equation}\label{q7}
\begin{aligned}
l_k&= \frac{2(1+\omega^{-k})}{8\cos^{2}(\frac{\pi k}{n-1})}S'v_k -8\cos^{2}(\frac{\pi k}{n-1})v_k         \\                                                             &=\frac{-4(1+\omega^{-k})(1+\omega^k)}{8\cos^{2}(\frac{\pi k}{n-1})}v_k -8\cos^{2}(\frac{\pi k}{n-1})v_k                                                             
 \\&= -\frac{4(2+\omega^{k}+\omega^{-k})  }{8\cos^{2}(\frac{\pi k}{n-1})}v_k-8\cos^{2}(\frac{\pi k}{n-1})v_k.
\end{aligned}
\end{equation}
Using the identity 
\[\omega^{k}+\omega^{-k}=2\cos(\frac{2 \pi k} {n-1}),\] in
$(\ref{q7})$,
\begin{equation}\label{q91}
l_k=
-\frac{(1+ \cos(\frac{2 \pi k} {n-1}) }{\cos^{2}(\frac{\pi k}{n-1})}v_k-8\cos^{2}(\frac{\pi k}{n-1})v_k.
\end{equation}
Since 
\[1+\cos(\frac{2\pi k}{n-1})=2\cos^{2}(\frac{\pi k}{n-1}),\]
we deduce
\begin{equation*}\label{q92}
l_k=(-2-8\cos^{2}(\frac{\pi k}{n-1}))v_k.
\end{equation*}
Hence,
\begin{equation*}
l_k=\theta_k v_k.
\end{equation*}
  By (\ref{q0}),
  \[Dq_{k}=(\nu_k,w_k',l_k')'
 =(0,\theta_{k}u_k', \theta_k v_k')'. \]
 Therefore,
\[Dq_k=\theta_k q_k.\]
This proves (i).

Let $i_1$ and $i_2$ be any two distinct integers in $\{1,\dotsc,n-2\}$.
We show that \[\langle{q_{i_1}}, {q_{i_2}}\rangle =0.\] 
Since
$q_{k}=(0,u_{k}',v_{k}')'$, we have
 \[\langle q_{i_1}, {q_{i_2}}\rangle =  \langle u_{i_1},u_{i_2}\rangle + \langle v_{i_1},v_{i_2}\rangle.\]
By   (K3), $\langle v_{i_1},v_{i_2}\rangle=0$.  
  Since  $u_k=\frac{-1}{8\cos^{2}(\frac{\pi k}{n-1})}(Sv_{k})$, 
  \begin{equation}\label{p1}
 \langle{u_{i_1}}, {u_{i_2}} \rangle=\frac{1}{64\cos^{2}(\frac{i_1 \pi }{n-1})\cos^{2}(\frac{i_2 \pi }{n-1})} \langle Sv_{i_1},Sv_{i_2} \rangle.
\end{equation}
 By (K3) and Lemma \ref{Z2}, 
  \[Sv_{k}=-2(1+\omega^{-k})v_{k}.\]   
Using this in (\ref{p1}), we have
\[  \langle{u_{i_1}}, {u_{i_2}} \rangle=\frac{4(1+\omega^{i_2})(1+\omega^{-i_1})}{64(\cos^{2}\frac{i_1 \pi }{n-1})(\cos^{2}\frac{i_2 \pi }{n-1})} \langle v_{i_1},v_{i_2} \rangle.\]
As $v_{i_1}$ and $v_{i_2}$ are orthogonal, it follows that $  \langle{u_{i_1}}, {u_{i_2}} \rangle=0$. Therefore,
\[\langle{q_{i_1}}, {q_{i_2}}\rangle =0.\]
This proves (ii).

Since $x_{j}=(\alpha_{j},\beta_{j}\1_{n-1}',\1_{n-1}' )^{'}$ and 
  $q_{k}:=(0,u_k',v_k')'$, we see that 
    \[\langle x_j,q_k\rangle=\langle\beta_{j}\1_{n-1},u_k\rangle+\langle\1_{n-1},v_k\rangle~~\mbox{and}~~\langle\1,q_k\rangle=\langle\1_{n-1},u_k\rangle+\langle\1_{n-1},v_k\rangle.\]
By  (\ref{1'uk=0}), $\langle \1_{n-1},u_k\rangle =0$ and
 $\langle\1_{n-1},v_k\rangle =0$.
 Hence $\langle x_j,q_k\rangle =0$ and $\langle \1,q_k\rangle=0$.
The proof of (iii) and (iv) is complete.
\end{proof}

\subsection{Computation of $L_{even}$}
Let $P$ be the orthogonal projection onto $\1^{\perp}$.
We now show that
\[L_{even}=(-\frac{1}{2} PDP)^{\dag}, \]
where $L_{even}$ is given in (\ref{speciallap}).
Put $U:=-\frac{PDP}{2}$.    
We have seen that the non-zero eigenvalues and the corresponding eigenvectors of $D$ are 
given by
\[D x_{j}=\lambda_{j}x_j~~~(j=1,2)~~\mbox{and}~~
Dq_{k}=\theta_{k} q_{k}~~(k=1,\dotsc,n-2). \]
Define
\begin{equation}\label{normxq}
\begin{aligned}
\x_{j}:&=\frac{x_j}{\|x_j\|}~~j=1,2;\\
\q_{k}:&=\frac{q_k}{\|q_k\|}~~k=1,\dotsc,n-2. 
\end{aligned}
\end{equation}
To compute $U^\dag$,
we begin with an elementary lemma.

\begin{lemma} \label{d+}
Let $a_{j}:=\langle\1,\x_{j}\rangle$. Then
\[D^{\dagger}\1 =\sum_{j=1}^{2} a_{j} \frac{\x_{j}}{\la_{j}}  
~~\mbox{and}~~
\1' D^\dagger \1=\frac{2}{n-1}.
\]
\end{lemma}
\begin{proof}
 By spectral decomposition, 
\begin{equation*}
D=\sum_{j=1}^{2}{\lambda_{j} {\x_j}{\x_j}^{'}}+                   \sum_{k=1}^{n-2}  {\theta_{k} {\q_k}{\q_k}^*}.
\end{equation*}
This implies
\begin{equation}\label{a2}
D^\dagger=\sum_{j=1}^{2}{\frac{1}{\lambda_{j}} {\x_{j}} {\x_{j}}^{'}}+                   \sum_{k=1}^{n-2}  {\frac{1}{\theta_{k}}  {\q_{k}}\q_{k}^\ast}.
\end{equation}
In view of item (iv) in Theorem $\ref{ee}$, $\langle \1,\q_k\rangle =0$. Thus, if
$a_{j}:=\langle\1,\x_{j}\rangle $, then 
\begin{equation*}
 D^{\dagger}\1 =\sum_{j=1}^{2}   {a_j}{\frac{\x_{j}}{\lambda_{j}}   
   }.
 \end{equation*} 
 Define \[ w:=(\frac{3-n}{n-1},\frac{1}{n-1}\1_{n-1},\0)'.\] Then, 
\[\langle \1,w \rangle=\frac{2}{n-1}~~\mbox{{and}}~~
 Dw=\1 .\] So, $D D^{\dag} \1=\1$. This implies that
 $w - D^{\dag} \1$ is in the null-space of $D$. By (\ref{1fi=0}),
 \[\langle \1,w-D^{\dag} \1 \rangle =0.\]  Therefore, 
 \[\langle \1,w \rangle   = \langle \1,D^{\dag}\1 \rangle=\frac{2}{n-1}.\]
 This completes the proof.
 \end{proof}
\begin{lemma}\label{qk}
 \[U^{\dagger} \q_k = -\frac{2}{\theta_k}\q_k~~~~k=1,\dotsc,n-2.\]
\end{lemma}
\begin{proof}
By the result in (K4),
\begin{equation*}
U^{\dagger}=-2D^{\dagger}+2\frac{ (D^\dagger\1)(D^\dagger\1)'}{(\1'D^\dagger\1)}.
\end{equation*}
Using (\ref{a2}) and Lemma \ref{d+},
\begin{equation*}
U^{\dagger}=\sum_{j=1}^{2}{\frac{-2}{\lambda_{j}}  {\x_{j}}  {\x_{j}} ^{'}}+                   \sum_{k=1}^{n-2}  {\frac{-2}{\theta_{k}} {\q_{k}}{\q_{k}}^{\ast}} +  ({n-1})(\sum_{j=1}^{2}   {a_j}{\frac{\x_{j}}{\lambda_{j}})(\sum_{j=1}^{2}   {a_j}{\frac{\x_{j}}{\lambda_{j}}  } })'.
\end{equation*}
 By Theorem \ref{ee}, 
 $\langle \x_{j},\q_{k}\rangle=0$ for all $j$ and $k$ and $\langle\q_{i_1},\q_{i_2}\rangle =0$ for all $i_1 \neq i_2$. Hence,
 \[ U^{\dagger} { \q_k} = -\frac{2}{\theta_k}{\q_k}.\]   
 The proof is complete.
\end{proof}

As $\rank(D)=n$, $\rank(U)=n-1$. In Lemma \ref{qk}, we have computed $n-2$
non-zero eigenvalues and the corresponding eigenvectors of $U$. In the next result, we compute the 
remaining non-zero eigenvalue and the corresponding eigenvector of $U$.
\begin{lemma}\label{y}
Let $y=\large(1,\frac{n-2}{3(n-1)} \1_{n-1}',-\frac{n+1}{3(n-1)}\1_{n-1}'\large )^{'}$. Then,
 \[U^\dagger y=\frac{2n-1}{n+4}y.\]
\end{lemma}
\begin{proof}
Define \[r:=\frac{n-2}{3(n-1)},~~~s:=-\frac{n+1}{3(n-1)},~~\gamma:= \frac{2(n+4)}{(2n-1)}.\] We now show that
\[U^{\dag}y=\frac{2n-1}{n+4}y, \] where
\[y=(1,r \1_{n-1}',s\1_{n-1}')^{'}.\]

 Recall that
 \begin{equation*} 
D=\left[
\begin{array}{ccccccccc}
0 & \1_{n-1}' & 2 \1_{n-1}'\\
\1_{n-1} & 2(J_{n-1}-I_{n-1}) & S\\
2 \1_{n-1} & S' & T
\end{array}
\right],
\end{equation*}
where $S=\cir(1,\underbrace{3,\dotsc,3}_{n-3},1)$ and $T=\cir(0,2,\underbrace {4,\dotsc,4}_{n-4},2)$.

Since \[J_{n-1}\1_{n-1}=(n-1)\1_{n-1},~~S\1_{n-1}=(3n-7)\1_{n-1}~\mbox{and}~ T\1_{n-1}=4(n-3)\1_{n-1},\]
by a direct computation, we get
\[Dy=\left[ \begin{array}{cccc}
(r+2s)(n-1)\\\
\1_{n-1}+2r(n-2)\1_{n-1}+s(3n-7)\1_{n-1}\\2\1_{n-1}+r(3n-7)\1_{n-1}+4s(n-3)\1_{n-1}
\end{array}\right],\]
and
\[ \gamma y   = \frac{2(n+4)}{(2n-1)}                               \left[ \begin{array}{cccc}
1\\
r\1_{n-1}\\
s\1_{n-1}
\end{array}\right]. \]
So,
\begin{equation} \label{y1}
Dy+\gamma y=  \frac{(n+4)(-2n+7)}{3(2n-1)}\1.
\end{equation}
Since \[\langle\1,y\rangle=1+\frac{(n-2)(n-1)}{3(n-1)}-\frac{(n+1)(n-1)}{3(n-1)}=0,\] 
and $P$ is the orthogonal projection onto $\1^{\perp}$, we get
\begin{equation}\label{y2222}
Py=y.
\end{equation} 
In view of (\ref{y1}) and (\ref{y2222}), 
\begin{equation*}
PDy+\gamma Py=PDPy+\gamma y=0.
\end{equation*} 
This gives 
\begin{equation*}\label{y4}
PDP y=-\gamma y.
\end{equation*} 

Since $U=-\frac{1}{2}PDP,$
\[U^\dagger y=\frac{2}{\gamma}y=\frac{2n-1}{n+4}y.\] The proof is complete.
 \end{proof}

\begin{lemma} \label{specU}
\[ U^{\dagger}=A
-\sum_{k=1}^{n-2}  {(\frac{2}{\theta_{k}}) }   \frac{ {q_{k}}  {q_{k}} ^{\ast }   }{ { 
{\langle{q_{k}}},{q_{k}}}\rangle},\]
where $A$ is defined in (\ref{A}).
\end{lemma}
\begin{proof}
As in Lemma $\ref{y}$, let \[y=(1,\frac{(n-2)}{3(n-1)}\1_{n-1}', \frac{-(n+1)}{3(n-1)}\1_{n-1}')^{'}.\] Now, 
\[  \|y\|^2=\langle y,y\rangle=\frac{(2n-1)(n+4)}{9(n-1)}.\]
 From Lemma \ref{qk}, for each $k=1,\dotsc,n-2$,
 \[U^{\dag} \q_{k}=-\frac{2}{\theta_k} \q_{k},  \]
 where
 $\theta_k={-8\cos^{2}(\frac{\pi k}{n-1})-2}$.
 From Lemma \ref{y}, 
\[U^{\dag}y= \frac{(2n-1)}{(n+4)}y. \] 
 Define
 \[\la:=\frac{2n-1}{n+4}. \]
 Then,
 \[\frac{\la}{\|y\|^{2}}=\frac{9(n-1)}{(n+4)^{2}}. \]
 Since $\rank(U)=n-1$, by
spectral decomposition, 
\begin{equation}\label{l1111}
\begin{aligned}
 U^{\dagger}&={\lambda}\frac{yy'}{{\|y\|}^2}-\sum_{k=1}^{n-2} {\frac{2}{\theta_{k}}  {\q_{k}}  {\q_{k}}^{\ast}} \\
 &=\frac{9(n-1)}{(n+4)^2}yy'-\sum_{k=1}^{n-2}  {(\frac{2}{\theta_{k}}) }   \frac{ {q_{k}}  {q_{k}} ^{\ast }   }{ { 
{\langle{q_{k}}},{q_{k}}}\rangle}. \\
 \end{aligned}
\end{equation}
Direct multiplication gives
\begin{equation}\label{av2222}
yy'=\left[ \begin{array}{cccc}
1 & \frac{(n-2)}{3(n-1)}\1_{n-1}' &\frac{-(n+1)}{3(n-1)}\1_{n-1}' \\
\frac{(n-2)}{3(n-1)} \1_{n-1}& \frac{(n-2)^2}{9(n-1)^2}J_{n-1} &\frac{-(n-2)(n+1)}{9(n-1)^2}J_{n-1} \\
\frac{-(n+1)}{3(n-1)}\1_{n-1} & \frac{-(n-2)(n+1)}{9(n-1)^2}J_{n-1} & \frac{(n+1)^2}{9(n-1)^2}J_{n-1}
\end{array}\right].
\end{equation}
We note from (\ref{A}) that \[A=\frac{9(n-1)}{(n+4)^2}yy'.\]
So,
\[ U^{\dagger}=A
-\sum_{k=1}^{n-2}  {(\frac{2}{\theta_{k}}) }   \frac{ {q_{k}}  {q_{k}} ^{\ast }   }{ { 
{\langle{q_{k}}},{q_{k}}}\rangle}.\]
The proof is complete.
\end{proof}

\begin{lemma}\label{qkvalue} 
Fix $k \in \{1,\dotsc,n-2\}$. Then,
${-(\frac{2}{\theta_{k}}) }\frac{ {q_{k}}  {q_{k}} ^{\ast }   }{ { 
{\langle{q_{k}}},{q_{k}}}\rangle}$ is equal to   \[
\frac{1}{(n-1)(2\cos(\frac{\pi k}{n-1})+ \frac{1}{                                                        2\cos(\frac{\pi k}{n-1})})^{2}}\left[ \begin{array}{cccc}
 0&\0&\0\\
\0& 
 \frac{1}{4\cos^{2}(\frac{\pi k}{n-1})} v_k v_k^{\ast}
 &
  \frac{1+\omega^{-k}}{                                                        4\cos^{2}(\frac{\pi k}{n-1})}v_k v_k^{\ast}\\
\0& \frac{1+\omega^{k}}{                                                        4\cos^{2}(\frac{\pi k}{n-1})}v_k v_k^{\ast}
&v_k v_k^{\ast} \end{array}\right],\]
where \[v_k=(1,\omega^{k},\dotsc,\omega^{(n-2)k})'.\]
\end{lemma}

\begin{proof}
Since  $q_k=(0,u_k',v_k')'$ and  $\theta_k={-8\cos^{2}(\frac{\pi k}{n-1})-2}$,
\[{-\frac{\theta_{k}}{2} }\langle{{q_{k}}}  ,    {q_{k}}\rangle= {(4\cos^{2}(\frac{\pi k}{n-1})+1)}(\langle u_{k},u_{k}\rangle+\langle v_{k},v_{k}\rangle)  . \]
Substituting $u_k=(  \frac{-1}{                                                        8\cos^{2}(\frac{\pi k}{n-1})})Sv_{k}$ in the above equation, 
\[{-\frac{\theta_{k}}{2} }\langle{{q_{k}}}      ,{q_{k}}\rangle=(4\cos^{2}(\frac{\pi k}{n-1})+1)(  \frac{1}{                                                      64\cos^{4}(\frac{\pi k}{n-1})}  \langle Sv_{k}, Sv_{k}\rangle+\langle v_{k},v_{k} \rangle).\]

By  Lemma $\ref{Z2}$ and (K3),
 \[Sv_k=-2(1+\omega^{-k})v_k.\]  
 Hence
\[{-\frac{\theta_{k}}{2} }\langle{{q_{k}}},      {q_{k}}\rangle= (4\cos^{2}(\frac{\pi k}{n-1})+1)(  \frac{(2+\omega^{k}+\omega^{-k}  )}{                                                        16\cos^{4}(\frac{\pi k}{n-1})  }\langle v_{k},v_{k}\rangle+\langle v_{k},v_{k} )\rangle.\]

By the identity $\omega^{k}+\omega^{-k}=2\cos(\frac{2 \pi k} {n-1})$, 
\[{-\frac{\theta_{k}}{2} }\langle{{q_{k}}},      {q_{k}}\rangle= (4\cos^{2}(\frac{\pi k}{n-1})+1)(  \frac{(1+\cos(\frac{2 \pi k}{n-1} ))}{                                                        8\cos^{4}(\frac{\pi k}{n-1})  }\langle v_{k},v_{k}\rangle+\langle v_{k},v_{k} )\rangle   ).\]

Since $v_k=(1,\omega^k,\dotsc,\omega^{(n-2)k})'$, \[\langle v_{k},v_{k}\rangle=n-1.\]
 Using this and the identity, 
   \[1+\cos(\frac{2\pi k}{n-1})=2\cos^{2}(\frac{\pi k}{n-1}),\]
  we get

 \begin{equation}\label{u2}
  \begin{aligned}
  {-\frac{\theta_{k}}{2} }\langle q_k, q_k \rangle&=(4\cos^{2}(\frac{\pi k}{n-1})+1)( \frac{\cos^{2}(\frac{\pi k}{n-1})}{                                                        4\cos^{4}(\frac{\pi k}{n-1})  }(n-1)+(n-1))
  \\
  &=(n-1)(2\cos(\frac{\pi k}{n-1})+ \frac{1}{                                                        2\cos(\frac{\pi k}{n-1})})^{2}.
   \end{aligned}
   \end{equation}

Since $q_k=(0,u_k',v_k')'$,
\[{{q_{k}}}  {q_{k}}^{\ast} =
\left[ \begin{array}{cccc} 
 0&\0&\0\\
\0&u_{k}u_{k}^{\ast} & u_{k}v_{k}^{\ast}\\
\0& v_{k}u_{k}^{\ast}& v_{k}v_{k}^{\ast}
 \end{array}\right].\]
 
Substituting $u_k=(  \frac{-1}{                                                        8\cos^{2}(\frac{\pi k}{n-1})})Sv_{k} $ and using $Sv_k=-2(1+\omega^{-k})v_k$, we get 

   \[
{{q_{k}}}  {q_{k}}^{\ast}
=\left[ \begin{array}{cccc}
 0&\0&\0\\
\0& 
 \frac{4(2+\omega^{-k}+\omega^{k})}{                                                        64\cos^{4}(\frac{\pi k}{n-1})  } v_{k}v_{k}^{\ast}
 &
  \frac{1+\omega^{-k}}{                                                        4\cos^{2}(\frac{\pi k}{n-1})  }v_{k}v_{k}^{\ast}\\
\0& \frac{1+\omega^{k}}{                                                        4\cos^{2}(\frac{\pi k}{n-1})  }v_{k}v_{k}^{\ast}
&v_{k}v_{k}^{\ast} \end{array}\right].
\]
 
As $\omega^{k}+\omega^{-k}=2\cos(\frac{2 \pi k} {n-1})$ and $1+\cos(\frac{2\pi k}{n-1})=2\cos^{2}(\frac{\pi k}{n-1})$, 

\begin{equation}\label{u1}
{{q_{k}}}  {q_{k}}^{\ast}
=\left[ \begin{array}{cccc}
 0&\0&\0\\
\0& 
 \frac{1}{4\cos^{2}(\frac{\pi k}{n-1})}v_k v_k^{\ast}
 &
  \frac{1+\omega^{-k}}{                                                        4\cos^{2}(\frac{\pi k}{n-1})}v_k v_k^{\ast}\\
\0& \frac{1+\omega^{k}}{                                                        4\cos^{2}(\frac{\pi k}{n-1})}v_k v_k^{\ast}
&v_k v_k^{\ast} \end{array}\right].
\end{equation}
Thus, from (\ref{u2}) and (\ref{u1}), 
$${-(\frac{2}{\theta_{k}}) }\frac{ {q_{k}}  {q_{k}} ^{\ast }   }{ { 
{\langle{q_{k}}},{q_{k}}}\rangle}   $$ is equal to
\begin{equation*}\label{u3}
\frac{1}{(n-1)(2\cos(\frac{\pi k}{n-1})+ \frac{1}{                                                        2\cos(\frac{\pi k}{n-1})})^{2}}\left[ \begin{array}{cccc}
 0&\0&\0\\
\0& 
 \frac{1}{4\cos^{2}(\frac{\pi k}{n-1})} v_k v_k^{\ast}
 &
  \frac{1+\omega^{-k}}{                                                        4\cos^{2}(\frac{\pi k}{n-1})}v_k v_k^{\ast}\\
\0& \frac{1+\omega^{k}}{                                                        4\cos^{2}(\frac{\pi k}{n-1})}v_k v_k^{\ast}
&v_k v_k^{\ast} \end{array}\right].
\end{equation*}
The proof is complete.
\end{proof}
\begin{lemma}\label{dbk}
Define
\[\Delta_k:=-(\frac{2}{\theta_{k}})\re\frac{ {q_{k}}  {q_{k}} ^{\ast }   }{ { 
{\langle{q_{k}}},{q_{k}}}\rangle}    ~~~~~k=1,\dotsc,n-2.
\]
Then,
\[\sum_{k=1}^{n-2} \Delta_k=\sum_{k=1}^{\frac{n-2}{2}}B_k,\]
where $B_k$ are defined in (\ref{bkdef}).
\end{lemma}
\begin{proof}
By $(\ref{bkdef})$, if $k \in \{1,\dotsc,n-2\}$, then
\[B_{k}:=  \frac{2}{(n-1)(2\phi_k+ \frac{1}{                                                        2\phi_k})^{2}} 
\left[
\begin{array}{ccc} 
 0&\0&\0\\
\0& \frac{C_k}{4\phi_{k  }^{2}}& \frac{C_k+\widetilde{C_{k}}}{ 4\phi_k^{2}  }\\ 
\0 & \frac{C_k+\widetilde{C_{k}}^{'}}{ 4\phi_{k}^{2}} & C_k \\
\end{array}
\right],
\]
where
\[\phi_k= \cos(\frac{\pi k}{n-1})~~~k=1,\dotsc,{n-2};\]
\[C_{k}=\cir(1,\cos(\frac{2\pi k }{n-1}),\dotsc,\cos(\frac{2 \pi (n-2)  k }{n-1}));~~\mbox{and}\]
 \[\widetilde{C_{k}}= \cir(\cos(\frac{2 \pi k}{n-1}),\dotsc,\cos(\frac{2 \pi (n-2) k }{n-1}),1).\] 
 By a direct verification, we see that $(r,s)^{\rm th}$ entries of $C_k$
 and $\widetilde{C_{k}}$ are 
 \begin{equation}\label{t1}
C_{k}(r,s)= \cos{(\frac{2\pi}{n-1}(r-s)k)} ~~\mbox{and}~~ \widetilde{C_{k}}(r,s)=\cos{(\frac{2\pi}{n-1}(r-s-1)k)}.
\end{equation}
Define $T_k:= v_k v_k^{\ast}.$
The $r^{\rm th}$ entry of $v_k$ is 
 $v_{k}(r)=\omega^{k(r-1)}$.
Hence, the $(r,s)^{\rm th}$ entry of $T_{k}$ is given by
\begin{equation}\label{tk}
    T_{k}(r,s)=v_{k}(r)v_{k}^{\ast}(s)=\omega^{k(r-1)}\omega^{-k(s-1)}= \omega^{(r-s)k}.
\end{equation} 
We note that
\[\re\omega^{(r-s)k}=\cos{(\frac{2\pi}{n-1}(r-s)k)}.\] 

So,
\[\re T_k(r,s)=\re \omega^{(r-s)k}=\cos{(\frac{2\pi}{n-1}(r-s)k)}.\]

In view of the first equation in (\ref{t1}), it now follows that
\begin{equation}\label{ck1}
\re T_k =C_k.
\end{equation}


Using (\ref{tk}),
\[(1+\omega^{-k})T_k(r,s)=(1+\omega^{-k})\omega^{(r-s)k}=\omega^{(r-s)k}+\omega^{(r-s-1)k}.\]
Hence 
\begin{equation*}
\begin{aligned}
\re{\large(1+\omega^{-k})T_k(r,s)\large} &=\re \omega^{(r-s)k} +\re \omega^{(r-s-1)k}  \\
&=\cos{(\frac{2\pi}{n-1}(r-s)k)}+
\cos{(\frac{2\pi}{n-1}(r-s-1)k)}.
\end{aligned}
\end{equation*}
In view of (\ref{t1}), 
\begin{equation}\label{ck2}
 \re(1+\omega^{k})T_k  = C_k+\widetilde{C_k}.
\end{equation}
 Let $F_k$ be the $(n-1) \times (n-1)$ matrix with $(r,s)^{\rm th}$ entry equal to
 \begin{equation}\label{f1}
 F_k (r,s)=\cos{(\frac{2\pi}{n-1}(r-s+1)k)}.
 \end{equation}
Again using (\ref{tk}), we see that
\[(1+\omega^{k})T_k(r,s)=(1+\omega^{k})\omega^{(r-s)k}=\omega^{(r-s)k}+\omega^{(r-s+1)k}.\]

 So,
 \begin{equation*}
 \begin{aligned}
\re\large(1+\omega^{k})T_k(r,s)\large &=\re \omega^{(r-s)k} +\re \omega^{(r-s+1)k}  \\
&=\cos{(\frac{2\pi}{n-1}(r-s)k)}+
\cos{(\frac{2\pi}{n-1}(r-s+1)k)}.
\end{aligned}
\end{equation*}

 By the first equation of (\ref{t1}) and by  (\ref{f1}),
\begin{equation}\label{ck3}
 \re (1+\omega^{k})T_k   =C_k +F_k.
\end{equation}

Since \[\cos{(\frac{2\pi}{n-1}(r-s+1)k)}=\cos{(\frac{2\pi}{n-1}(s-r-1)k)},\]
from the second equation in $(\ref{t1})$, we find that
\[\widetilde{C_k}(s,r)=\cos{(\frac{2\pi}{n-1}(r-s+1)k)}. \]
By the definition of $F_k$, we now have
\[F_k(r,s)=\widetilde{C_k}(s,r).\]
So,
\[F_k= \widetilde{C_k}^{'}.\] 
Thus, (\ref{ck3}) simplifies to
\begin{equation}\label{ckt}
 \re (1+\omega^{k})T_k  =C_k +\widetilde{C_k}^{'}.
 \end{equation}
In view of Lemma \ref{qkvalue}, 
\[{-(\frac{2}{\theta_{k}}} \frac{1}{\langle q_k,q_k \rangle})q_{k}q_{k}^{*}
=\frac{1}{(n-1)(2\cos(\frac{\pi k}{n-1})+ \frac{1}{                                                        2\cos(\frac{\pi k}{n-1})})^{2}}\left[ \begin{array}{cccc}
 0&\0&\0\\
\0& 
 \frac{1}{4\cos^{2}(\frac{\pi k}{n-1})} T_k
 &
  \frac{1+\omega^{-k}}{                                                        4\cos^{2}(\frac{\pi k}{n-1})}T_k\\
\0& \frac{1+\omega^{k}}{                                                        4\cos^{2}(\frac{\pi k}{n-1})}T_k
&T_k\end{array}\right].\]
As in Section \ref{speciallaplacianeven}, let \[\phi_k= \cos(\frac{\pi k}{n-1})~~~k=1,\dotsc,{n-2}.\]  Now,
\begin{equation} \label{qkqk}
{-(\frac{2}{\theta_{k}}} \frac{1}{\langle q_k,q_k \rangle})q_{k}q_{k}^{*}
=\frac{2}{(n-1)(2\phi_k+ \frac{1}{                                                        2\phi_k})^{2}} 
\left[ \begin{array}{cccc}
 0&\0&\0\\
\0& 
 \frac{1}{4\phi_k^2} T_k
 &
  \frac{1+\omega^{-k}}{                                                        4 \phi_k^2}T_k\\
\0& \frac{1+\omega^{k}}{                                                        4 \phi_k^2  }T_k
&T_k\end{array}\right].
\end{equation}
Since
\[\Delta_k=-(\frac{2}{\theta_{k}})   \re{ }\frac{ {q_{k}}  {q_{k}} ^{\ast }   }{ { 
{\langle{q_{k}}},{q_{k}}}\rangle}   ,
\]
by using (\ref{ck1}), (\ref{ck2}) and (\ref{ckt}) in (\ref{qkqk}), we get
\begin{equation}\label{deltak}
\sum_{k=1}^{n-2} \Delta_k=\sum_{k=1}^{n-2}    \frac{1}{(n-1)(2\phi_k    + \frac{1}{                                                        2 \phi_k  })^{2}}  \left[ \begin{array}{cccc} 0&\0'&\0'\\
\0&\frac{C_k}{                                                        4\phi_k^2   }&\frac{C_k+\widetilde{C_{k}}}{                                                        4 \phi_k^2  }\\ \0&\frac{C_k+\widetilde{C_{k}}^{'}}{                                                        4 \phi_k^2     }&C_k
\end{array}\right].
\end{equation}
We recall from (\ref{bkdef}) that 
\[B_k=    \frac{2}{(n-1)(2\phi_k    + \frac{1}{                                                        2 \phi_k  })^{2}}  \left[ \begin{array}{cccc} 0&\0'&\0'\\
\0&\frac{C_k}{                                                        4\phi_k^2   }&\frac{C_k+\widetilde{C_{k}}}{                                                        4 \phi_k^2  }\\ \0&\frac{C_k+\widetilde{C_{k}}^{'}}{                                                        4 \phi_k^2     }&C_k
\end{array}\right].\]
Hence by (\ref{deltak}), 
\begin{equation}\label{dkbk}
\sum_{k=1}^{n-2} \Delta_k=\frac{1}{2}   \sum_{k=1}^{n-2}B_k.
\end{equation}
Fix $k \in \{1,\dotsc,\frac{n-2}{2}\}$.
We now claim that \[B_k=B_{n-1-k}.\] 
To prove this, we need to show that 
\[C_{n-1-k}=C_k,~~\widetilde{C}_{n-1-k}=\widetilde{C}_{k},\] \[\phi^2_{n-1-k}=\phi^2_{k}~~\mbox{and}~~(2\phi_{n-1-k}+\frac{1}{2\phi_{n-1-k}})^2=(2\phi_{k}+\frac{1}{2\phi_{k}})^2.\]
Since
\begin{align*}
C_{n-1-k}(r,s)&= \cos{(\frac{2\pi}{n-1}(r-s)(n-1-k))}\\
&=\cos{(2\pi (r-s)-\frac{2\pi}{n-1}(r-s)k)}\\&
=C_k(r,s),
\end{align*}
it follows that
$C_{n-1-k}=C_k$. 
Similarly,
\begin{align*}
\widetilde{C}_{n-1-k}(r,s)&= \cos{(\frac{2\pi}{n-1}(r-s-1)(n-1-k))}\\
&=\cos{(2\pi (r-s-1)-\frac{2\pi}{n-1}(r-s-1)k)}\\&=\widetilde{C}_k(r,s).
\end{align*}
Hence $\widetilde{C}_{n-1-k}=\widetilde{C}_{k}$.

We note that
\begin{align*}
\phi^2_{n-1-k}&=\cos^2(\frac{\pi}{n-1}(n-1-k))\\
&=\cos^2(\pi-\frac{\pi}{n-1}k)\\
&=\cos^2(\frac{\pi}{n-1}k)=\phi^2_{k}.
\end{align*}
Finally, since
$\phi^2_{n-1-k}=\phi^2_{k}$, we obtain
\begin{align*}
(2\phi_{n-1-k}+\frac{1}{2\phi_{n-1-k}})^2&=
4\phi_{n-1-k}^2+\frac{1}{4\phi_{n-1-k}^2}+2\\&=4\phi_{k}^2+\frac{1}{4\phi_{k}^2}+2\\&=(2\phi_{k}+\frac{1}{2\phi_{k}})^2.
\end{align*}

Hence,  $B_k=B_{n-1-k} $. Thus, (\ref{dkbk})
simplifies to
\begin{equation*}
\sum_{k=1}^{n-2} \Delta_k=\sum_{k=1}^{\frac{n-2}{2}}B_k.
\end{equation*}
This completes the proof.
\end{proof}
\begin{theorem}\label{lapev}
$L_{even}=   U^{\dagger}  $
\end{theorem}
\begin{proof}

By Lemma \ref{specU}, we have
\begin{equation}\label{r1}
 U^{\dagger}=A
-\sum_{k=1}^{n-2}  {(\frac{2}{\theta_{k}}) }  \frac{ {q_{k}}  {q_{k}} ^{\ast }   }{ { 
{\langle{q_{k}}},{q_{k}}}\rangle}.
\end{equation}
Since $U^{\dagger}$ and $A$ are real matrices,
from (\ref{r1}), 
\begin{equation}\label{r2}
U^{\dagger}= A+\sum_{k=1}^{n-2}(-\frac{2}{\theta_{k}})\re{ }\frac{ {q_{k}}  {q_{k}} ^{\ast }   }     { { 
{\langle{q_{k}}},{q_{k}}}\rangle}.
\end{equation}
Since $ \Delta_k={(-\frac{2}{\theta_{k}})\re }\frac{ {q_{k}}{q_{k}} ^{\ast }}{{ \langle{{q_{k}}},{q_{k}}\rangle}}$,
 we have
\begin{equation}\label{ll4}
U^{\dagger}=A+\sum_{k=1}^{n-2}\Delta_k.
\end{equation}
By Lemma \ref{dbk}, 
\[\sum_{k=1}^{n-2}\Delta_k=\sum_{k=1}^{\frac{n-2}{2}}B_k.\]
Thus,  (\ref{ll4}) becomes
\[U^{\dagger}=A+\sum_{k=1}^{\frac{n-2}{2}}B_k. \]
In view of  (\ref{speciallap}),
\[L_{even} =A+\sum_{k=1}^{\frac{n-2}{2}}B_k.\]
Hence, \[U^{\dagger}=L_{even}.\]
This completes the proof.
\end{proof}

\subsection{Computation of $D^\dag \1$}
We need to compute $D^\dag \1$ now.
\begin{lemma}\label{uvalue}
\[D^{\dagger}\1=(\frac{-3n+13}{(n^2+3n-4)}, \frac{-n+6}{(n^2+3n-4)}\1_{n-1}',\frac{1}{n+4}\1_{n-1}' )'.\]
\end{lemma}
\begin{proof}
First, we claim that 
\[D^{\dagger}\1=
(c, a\1_{n-1}', b\1_{n-1}')',\] where $a,b,c \in \rr$.
By Lemma \ref{d+},
\begin{equation} \label{dd1}
D^{\dagger}\1 =\sum_{j=1}^{2} a_{j} \frac{\x_{j}}{\la_{j}},
\end{equation} 
where \[a_{j}=\langle \1,\x_{j}\rangle,~~\x_{j}=\frac{x_j}{\|x_j\|},~~\lambda_{j}=3n-8-(-1)^{j}\sqrt{5}\sqrt{n(2n-9)+12}. \]
We note from (K2) that 
\[x_j=(\alpha_j,\beta_j\1_{n-1}',\1_{n-1}')',\] where $\alpha_j$ and $\beta_j$
are some scalars.
Since (\ref{dd1}) can be expressed as
  \[D^{\dagger}\1= \sum_{j=1}^{2} a_{j} \frac{x_{j}}{\la_{j}{\|x_j\|}},  \] 
 it follows that 
\[D^{\dagger}\1
=(c, a\1_{n-1}', b\1_{n-1}')',\] 
 where $a,b,c$ are some scalars. This proves the claim.

We now show that
\[c=\frac{-3n+13}{(n^2+3n-4)},~~ a=\frac{-n+6}{(n^2+3n-4)},~~b=\frac{1}{n+4}.\]
Direct multiplication gives
 \[(D^{\dagger}\1)(D^{\dagger}\1 )' = \left[
\begin{array}{ccccccccc}
c^2 & ac\1_{n-1}' & cb\1_{n-1}'\\
ca\1_{n-1} & a^2J_{n-1} & abJ_{n-1}\\
bc \1_{n-1} & baJ_{n-1} & b^2J_{n-1}
\end{array}
\right].\]
Hence
\begin{equation}\label{dvec}
\tr~(D^{\dagger}\1)(D^{\dagger}\1 )'=c^2+(n-1)a^2+(n-1)b^2.
\end{equation}
By (K2), the non-zero eigenvalues of $D^\dagger $ are 
\[\frac{1}{\lambda_{1}},\frac{1}{\lambda_{2}},\frac{1}{ \theta_{1}},\dotsc,\frac{1}{ \theta_{n-2}},\]
where $\lambda_{1},\lambda_{2},\theta_{1},\dotsc,\theta_{n-2}$ are given in 
$(\ref{lambdaj})$ and $(\ref{thetak})$. 
So,
\begin{equation}\label{traced}
\tr~ D^\dagger=\sum_{i=1}^{2}\frac{1}{\lambda_{i}} +   \sum_{k=1}^{n-2}\frac{1}{\theta_{k}}.
\end{equation}
By Lemma $\ref{qk}$ and Lemma $\ref{y}$, 
$L_{even}$ has nonzero eigenvalues
\[\frac{2n-1}{n+4},~
\frac{-2}{ \theta_{1}},\dotsc,\frac{-2}{ \theta_{n-2}}.\]

We now claim that $\rank(L_{even})=n-1$. 
Note that $\rank(L_{even})=\rank(PDP)$. 
By equation (\ref{null})
and $(\ref{1fi=0})$, 
\[\{f_{1},\dotsc,f_{n-1},\1\} \subseteq \mbox{null-space}(PDP).\]
By Lemma  $\ref{qk}$ and Lemma $\ref{y}$, $L_{even}$ has at least $n-1$ non-zero eigenvalues. Hence, $L_{even}$ has exactly $n-1$ non-zero eigenvalues.
This shows that $\rank(L_{even})=n-1$. The proof of claim is complete.

So,
 \begin{equation}\label{traceu}
\tr ~L_{even}=\frac{2n-1}{n+4}+\sum_{k=1}^{n-2}\frac{-2}{\theta_{k}}. 
\end{equation}
By (K4),
\[D^{\dagger}=-\frac{1}{2} L_{even}+\frac{1}{\1'D^\dag \1}(D^\dag \1)(D^\dag \1)^{'}. \]
From Lemma \ref{d+}, we have $ \1'D^\dag \1=  \frac{2}{n-1}$.
Hence,
\begin{equation}\label{n-12}
D^{\dagger}=-\frac{1}{2}L_{even}+\frac{n-1}{2}(D^\dag \1)(D^\dag \1)^{'}. 
\end{equation}

Using (\ref{dvec}) and (\ref{traceu}) in (\ref{n-12}), we obtain
\begin{equation}\label{sumld}
\tr~ D^\dagger=-\frac{2n-1}{2(n+4)}+\sum_{k=1}^{n-2}\frac{1}{\theta_{k}}+\frac{n-1}{2}(c^2+(n-1)a^2+(n-1)b^2).
\end{equation}
By (\ref{traced}) and (\ref{sumld}), 


\begin{equation}\label{u11}
\frac{\lambda_{1}+\lambda_{2}}{\lambda_{1}\lambda_{2}}=\frac{-1}{2}\frac{(2n-1)}{(n+4)}+\frac{n-1}{2}(c^2+(n-1)a^2+(n-1)b^2).
\end{equation}
Since \[\lambda_{j}=3n-8-(-1)^{j}\sqrt{5}\sqrt{n(2n-9)+12},\]
 (\ref{u11}) leads to
\[\frac{2(3n-8)}{-(n^2+3n-4)}+\frac{(2n-1)}{2(n+4)}=\frac{n-1}{2}(c^2+(n-1)a^2+(n-1)b^2).\]
The above equation can be rewritten as
\begin{equation}\label{u22}
\frac{2n^2-15n+33}{2(n+4)(n-1)}=\frac{n-1}{2}(c^2+(n-1)a^2+(n-1)b^2).
\end{equation}
Recall that
\begin{equation*} 
D=\left[
\begin{array}{ccccccccc}
0 & \1_{n-1}' & 2 \1_{n-1}'\\
\1_{n-1} & 2(J_{n-1}-I_{n-1}) & S\\
2 \1_{n-1} & S' & T
\end{array}
\right],
\end{equation*}
where $S=\cir(1,\underbrace{3,\dotsc,3}_{n-3},1)$ and $T=\cir(0,2,\underbrace {4,\dotsc,4}_{n-4},2)$.

Since $DD^\dag \1=\1$, we have
\[\left[
\begin{array}{ccccccccc}
0 & \1_{n-1}' & 2 \1_{n-1}'\\
\1_{n-1} & 2(J_{n-1}-I_{n-1}) & S\\
2 \1_{n-1} & S' & T
\end{array}
\right]\left[
\begin{array}{ccccccccc}c\\a\1_{n-1}\\ b\1_{n-1}\end{array}
\right]=\left[
\begin{array}{ccccccccc}1\\ \1_{n-1}\\ \1_{n-1}\end{array}
\right].\]
Because \[J_{n-1}\1_{n-1}=(n-1)\1_{n-1},~~S\1_{n-1}=(3n-7)\1_{n-1}~\mbox{and}~ T\1_{n-1}=4(n-3)\1_{n-1},\]
we get,
\begin{equation}\label{u55}
\begin{aligned}
(n-1)(a+2b)&=1  \\
c+2a(n-2)+b(3n-7)&=1\\
2c+a(3n-7)+4b(n-3)&=1.
\end{aligned}
\end{equation}
By (\ref{u22}) and (\ref{u55}), we get 
\[c=\frac{-3n+13}{(n^2+3n-4)},~~ a=\frac{-n+6}{(n^2+3n-4)},~~b=\frac{1}{n+4}.\]
Thus,
 \[D^{\dagger}\1=(\frac{-3n+13}{(n^2+3n-4)}, \frac{-n+6}{(n^2+3n-4)}\1_{n-1}', \frac{1}{n+4}\1_{n-1}' )'.\] 
 The proof is complete.
  \end{proof}
\subsection{Proof of inverse formula}
We prove our main result now.
\begin{theorem}
Let $n \geq 4$ be an even integer. Then the Moore-Penrose inverse of 
$D$ is given by 
\[D^\dag = -\frac{1}{2}L_{even}+\frac{n-1}{2}uu',\]
where 
\[u= \huge(\frac{-(3n-13)}{(n^2+3n-4)},~ \frac{-(n-6)}{(n^2+3n-4)}\1_{n-1}',~\frac{1}{(n+4)}\1_{n-1}' \huge)'.\]
\end{theorem}

\begin{proof}Since $D$ is an EDM and $\1'D^\dag\1>0$, by the result in (K4),
\[D^{\dagger}=-\frac{1}{2}U^{\dagger}+\frac{1}{\1'D^\dag \1}(D^\dag \1)(D^\dag \1)^{'}. \]
Using Theorem  \ref{lapev}  and  Lemma \ref{uvalue},
we get
$L_{even}=U^\dag$ and $u=D^\dag \1$.
From Lemma \ref{d+},  we have $\1'D^\dag \1=\frac{2}{n-1}$.
Hence,
\[D^{\dagger}=-\frac{1}{2}L_{even}+\frac{n-1}{2}uu^{'}. \]
The proof is complete.
\end{proof}

\section{Moore-Penrose inverse of $D(G_{n})$ when $n$ is odd}
We now consider $G_n$ where $n$ is an odd integer.
As in even case, we define
a special odd Laplacian matrix now. This definition is motivated by numerical experiments.
\subsection{Special odd Laplacian} \label{spl.lap.odd}
Define \[\phi_k= \cos(\frac{\pi k}{n-1})~~~k=1,\dotsc,{n-2}. \]
Since $n-1$ is odd, 
$\phi_{\frac{n-1}{2}}=0$. For any other  $k$, $\phi_{k} \neq 0$. Define
\[\nabla:=\{1,\dotsc,n-2\} \smallsetminus \{\frac{n-1}{2}\}.  \]
Let
\begin{equation} \label{AA}
A:=\frac{9(n-1)}{(n+4)^2}
\left[
\begin{array}{ccccccccc}
1 & \frac{(n-2)}{3(n-1)}\1_{n-1}' &\frac{-(n+1)}{3(n-1)}\1_{n-1}' \\
\frac{(n-2)}{3(n-1)}\1_{n-1}& \frac{(n-2)^2}{9(n-1)^2}J_{n-1} &\frac{-(n-2)(n+1)}{9(n-1)^2}J_{n-1} \\
\frac{-(n+1)}{3(n-1)} \1_{n-1}& \frac{-(n-2)(n+1)}{9(n-1)^2}J_{n-1} & \frac{(n+1)^2}{9(n-1)^2}J_{n-1}
\end{array}
\right],
\end{equation}
\begin{equation}\label{bk2val}
B_{k}:=  \frac{2}{(n-1)(2\phi_k+ \frac{1}{                                                        2\phi_k})^{2}} 
\left[
\begin{array}{ccc} 
 0&\0&\0\\
\0& \frac{C_k}{4\phi_{k  }^{2}}& \frac{C_k+\widetilde{C_{k}}}{ 4\phi_k^{2}  }\\ 
\0 & \frac{C_k+\widetilde{C_{k}}^{'}}{ 4\phi_{k}^{2}} & C_k \\
\end{array}
\right]~~~k \in \nabla,
\end{equation}
\begin{equation}\label{h22}
H:=\frac{1}{n-1}\begin{bmatrix}  0&\0'&\0'\\
\0&T&O\\ \0&O&O\end{bmatrix},
\end{equation}
 where
 \[T(r,s)=(-1)^{r+s},\]
 \[C_{k}=\cir(1,\cos(\frac{2\pi k }{n-1}),\dotsc,\cos(\frac{2 \pi (n-2)  k }{n-1})),\]
 \[\widetilde{C_{k}}= \cir(\cos(\frac{2 \pi k}{n-1}),\dotsc,\cos(\frac{2 \pi (n-2) k }{n-1}),1).\] 
 The special odd Laplacian matrix for $G_n$ is defined by
 \begin{equation}\label{speciallapo}
L_{odd}:=A+H+ \sum_{k=1}^{\frac{n-3}{2}}B_k.
\end{equation}
\subsection{Inverse formula for $D(G_n)$}
Our main result in this section is the following.
\begin{theorem}\label{main2}
Let $n \geq 5$ be an odd integer. Then the Moore-Penrose inverse of 
$D(G_n)$ is given by 
\[D(G_n)^\dag = -\frac{1}{2}L_{odd}+\frac{n-1}{2}uu',\]
where 
\[u= \huge(\frac{-(3n-13)}{(n^2+3n-4)},~ \frac{-(n-6)}{(n^2+3n-4)}\1_{n-1}',~\frac{1}{(n+4)}\1_{n-1}' \huge)'.\]
\end{theorem}

\subsection{Illustration of the formula for $G_5$}
 The distance matrix of the gear graph 
$G_5$ is 
\[D:=D(G_5)= \left[
\begin{array}{ccccccccc}0& \1_{4}'&2\1_{4}'\\
\1_{4}&2(J_{4}-I_{4})& \cir( 1,3,3,1)\\
2\1_{4}&\cir(1,1,3,3)&\cir(0,2,4,2)
 \end{array}
\right].
\]
The Laplacian $L:=L_{odd}$ of $G_5$ is given by
\[L=
  \left[
\begin{array}{ccccccccc}
\frac{4}{9}& \frac{1}{9}\1_{4}'   &\frac{-2}{9}\1_{4}'\\
 \frac{1}{9}\1_{4}&\cir( \frac{1}{3}, \frac{-2}{9},   \frac{2}{9},\frac{   -2}{9})&\cir( 0,\frac{-1}{9},  \frac{-1}{9},0)\\
 \frac{-2}{9}\1_{4}'& \cir(0, 0,\frac{-1}{9},  \frac{-1}{9})
&\cir( \frac{ 2}{9},\frac{1}{9},0,  \frac{1}{9})
\end{array}
\right].
\]
The vector $u$ in Theorem \ref{main2} is
\[u=(\frac{-1}{18},\frac{1}{36}\1_4,\frac{1}{9}\1_{4}')'.\]
According to Theorem \ref{main2},
\[D^\dagger =-\frac{1}{2}L +{2}uu'.\]

The right hand side of the above equation is
\[ 
  \left[
\begin{array}{ccccccccc}
\frac{-35}{162}& \frac{-19}{324}\1_{4}'   &\frac{8}{81}\1_{4}'\\
 \frac{-19}{324}\1_{4}&\cir( \frac{-107}{648}, \frac{73}{648},   \frac{-71}{648},\frac{ 73}{648})&\cir( \frac{1}{162},\frac{5}{81},  \frac{5}{81}, \frac{1}{162})\\
 \frac{8}{81}\1_{4}'& \cir(\frac{1}{162}, \frac{1}{162},\frac{5}{81},  \frac{5}{81} )
&\cir( \frac{ -7}{81},\frac{-5}{162},\frac{2}{81},  \frac{-5}{162})
\end{array}
\right],
\]
which is the Moore-Penrose inverse of $D$.
To prove the formula in Theorem \ref{main2}, we now proceed to compute the non-zero eigenvalues and the corresponding eigenvectors of $D(G_n)$.
We shall use $E$ for $D(G_n)$ in the sequel.
 
\subsection{Eigenvalues and eigenvectors of $E$}
As in even case, we compute all the non-zero eigenvalues and the corresponding eigenvectors of $E$. This is partly done in
(K2). The nonzero eigenvalues of $E$ are given by $\lambda_1,\lambda_2$ and 
$\theta_1,\dotsc,\theta_{n-2}$ in (K2). Eigenvectors corresponding to 
$\lambda_{1}$ and $\lambda_{2}$ were already computed and we need to find eigenvectors
corresponding to $\theta_k$.
\begin{theorem}\label{eee}
 Let $k \in \{1,\dotsc,n-2\}$. 
Define
\[ q_{k}:=(0,u_k',v_k')',\]
where 
\[u_{k}=
\begin{cases}
\frac{-1}{8\cos^{2}(\frac{\pi k}{n-1})}S v_{k}&  k\in \nabla\\
(1,-1,\dotsc,1,-1)'&  k=\frac{n-1}{2}.
\end{cases}\]
and 
\[v_{k}=
\begin{cases}
(1,\omega^{k},\omega^{2k},\dotsc,\omega^{k(n-2)})'&  k\in \nabla\\
\0&  k=\frac{n-1}{2},
\end{cases}\]
\[
~S=\cir(1,\underbrace{3,\dotsc,3}_{n-3},1)~~\mbox{and}\]
\[\theta_{k}:=-8\cos^{2}(\frac{\pi k}{n-1})-2.\]
Then,
\begin{enumerate}
\item[\rm (i)] $E q_k=\theta_k q_k$.
\item[\rm (ii)]
The vectors $q_1,\dotsc,q_{n-2}$ are mutually orthogonal.
\item[\rm (iii)] $\langle x_j,  q_{k}\rangle=0$.
\item[\rm (iv)] $\langle \1, q_k\rangle=0$. 
\end{enumerate}
\end {theorem}
\begin{proof}
To prove 
\[Eq_k =\theta_k q_k ~~~k \in \nabla,\] one can use the same argument as given in Theorem \ref{ee}.
Suppose $k=\frac{n-1}{2}$.
We now  show that
\[Eq_{k}=\theta_{k} q_{k}=-2 q_{k}.\]
Recall from Definition \ref{dgn} that
\[E=\left[\begin{array}{ccc}
0 & \1_{n-1}^{'} & 2\1_{n-1}^{'}\\
\1_{n-1} & 2(J_{n-1}-I_{n-1}) & S\\
2\1_{n-1} & S^{'} & T
\end{array}\right].\]
By direct computation, we get
\begin{equation}\label{wn1}
Eq_k=\left[\begin{array}{ccc}
\1_{n-1}^{'}u_k
\\ 2(J_{n-1}-I_{n-1})u_k
\\S^{'} u_k
\end{array}\right].
\end{equation}
Since
\begin{equation*}\label{wn2}
\1_{n-1}^{'}u_k=0,
\end{equation*}
\begin{equation}\label{wn3}
2(J_{n-1}-I_{n-1})u_k=2\1_{n-1}
\1_{n-1}'u_k-2u_k=-2u_k.
\end{equation}
We recall that $S^{'}=\cir(1,1,\underbrace{3,3,\dotsc,3}_{n-3})$.  Because $n-1$ is even, we see that
\begin{equation}\label{wn4}
S^{'} u_k=(1-1+3-3\cdots+3-3) \1_{n-1}=\0.
\end{equation}
Using (\ref{wn3}) and  (\ref{wn4}) in  (\ref{wn1}), we obtain
\[E{{ q}}_k=-2{{ q}}_k.\]
Proof of (ii), (iii) and (iv) follow by a similar argument given in the proof of Theorem {\ref{ee}}.
\end{proof}

To this end, we have computed all the nonzero eigenvalues of
$E$. These are denoted by
 \[\lambda_1,\lambda_2,\theta_1,\dotsc,\theta_{n-2}.\]
Eigenvectors corresponding to $\lambda_{1}$ and $\lambda_{2}$ are $x_{1}$ and
$x_2$ given in (K2).
Eigenvectors corresponding to $\theta_1,\dotsc,\theta_{n-2}$ are
given by $q_{1},\dotsc,q_{n-2}$ in the previous theorem.
In the rest of the proof, we use $j$ to
denote an integer in $\{1,2\}$ and $k$ to denote an integer in
$\{1,\dotsc,n-2\}$.

\subsection{Computation of $L_{odd}$}

Let $P$ be the orthogonal projection onto $\{\1\}^{\perp}$. We now show that
\[L_{odd}=(-\frac{1}{2} PEP)^{\dag}. \]
Put $U:=-\frac{PEP}{2}$.    
Define
\[\x_{j}:=\frac{x_j}{\|x_j\|}
~~\mbox{and}~~\q_{k}:=\frac{q_k}{\|q_k\|}. \]

We now have
\[E \x_{j}=\lambda_{j}\x_j~~
\mbox{and}~~E\q_{k}=\theta_{k} \q_{k}. \]

To compute $U^\dag$ precisely, we begin with an elementary lemma.

\begin{lemma} \label{d1+}
\begin{enumerate}
\item[\rm (i)]
Let $a_{j}:=\langle \1,\x_{j}\rangle $. Then
\[E^{\dagger}\1 =\sum_{j=1}^{2} a_{j} \frac{\x_{j}}{\la_{j}}  
~~\mbox{and}~~
\1' E^\dagger \1=\frac{2}{n-1}.
\]
\item[\rm (ii)]
\[U^{\dagger} \q_k = -\frac{2}{\theta_k}\q_k.\]
\item[\rm (iii)]
Let $y=\large(1,\frac{n-2}{3(n-1)} \1_{n-1}',-\frac{n+1}{3(n-1)}\1_{n-1}'\large )^{'}$. Then,
 \[U^\dagger y=\frac{2n-1}{n+4}y.\]
 \item[\rm (iv)]
If $k\in\nabla$,
$${-(\frac{2}{\theta_{k}}) }\frac{ {q_{k}}{q_{k}} ^{\ast }   }{ {\langle {{q_{k}}},{q_{k}\rangle}}}$$ is equal to
\[\frac{1}{(n-1)(2\cos(\frac{\pi k}{n-1})+ \frac{1}{                                                        2\cos(\frac{\pi k}{n-1})})^{2}}\left[ \begin{array}{cccc}
 0&\0&\0\\
\0& 
 \frac{1}{4\cos^{2}(\frac{\pi k}{n-1})} v_k v_k^{\ast}
 &
  \frac{1+\omega^{-k}}{                                                        4\cos^{2}(\frac{\pi k}{n-1})}v_k v_k^{\ast}\\
\0& \frac{1+\omega^{k}}{                                                        4\cos^{2}(\frac{\pi k}{n-1})}v_k v_k^{\ast}
&v_k v_k^{\ast} \end{array}\right],\]
where \[v_k=(1,\omega^{k},\dotsc,\omega^{(n-2)k})'.\] 
\item[\rm (v)]
Define
\[\Delta_k:=-(\frac{2}{\theta_{k}})\re{ }\frac{ {q_{k}}{q_{k}} ^{\ast }}{{ \langle {{q_{k}}},{q_{k}}\rangle}}~~~~~k \in \nabla\   .
\]
Then,
\[ \sum_{k \in \nabla}   \Delta_k=\sum_{k=1}^{\frac{n-3}{2}}B_k,\]
where $B_k$ are defined in  (\ref{bk2val}).
 
 \end{enumerate}
\end{lemma}
\begin{proof}
The proof of (i) is similar to Lemma \ref{d+}, (ii) is similar to Lemma \ref{qk}, (iii) is similar to Lemma \ref{y}, (iv) is similar to Lemma \ref{qkvalue} and (v) is similar to Lemma \ref{dbk}.
\end{proof}

\begin{lemma} \label{specU2}
\[ U^{\dagger}=A+H
- \sum_{k \in \nabla }    {(\frac{2}{\theta_{k}}) }   \frac{ {q_{k}}  {q_{k}} ^{\ast }   }{ { 
{\langle {q_{k}}},{q_{k}}\rangle }},\]
where $A$  and $H$ are  defined in (\ref{AA}) and (\ref{h22}) respectively.
\begin{proof}

As in item (iii) of Lemma $\ref{d1+}$, let \[y=(1,\frac{(n-2)}{3(n-1)}\1_{n-1}', \frac{-(n+1)}{3(n-1)}\1_{n-1}')^{'}.\] Now, 
\[  \|y\|^2=\langle y,y\rangle =\frac{(2n-1)(n+4)}{9(n-1)}.\]
 In view of item (ii) of Lemma \ref{d1+}, for each $k=1,\dotsc,n-2$,
 \[U^{\dag} \q_{k}=-\frac{2}{\theta_k} \q_{k},  \]
 where
 $\theta_k={-8\cos^{2}(\frac{\pi k}{n-1})-2}$.
 By item (iii) of Lemma \ref{d1+}, 
\[U^{\dag}y= \frac{(2n-1)}{(n+4)}y. \] 
 Define
 \[\la:=\frac{2n-1}{n+4}. \]
 Then,
 \begin{equation}\label{y3}
 \frac{\la}{\|y\|^{2}}=\frac{9(n-1)}{(n+4)^{2}}.
 \end{equation}
 Direct multiplication gives
\begin{equation}\label{av}
yy'=\left[ \begin{array}{cccc}
1 & \frac{(n-2)}{3(n-1)}\1_{n-1}' &\frac{-(n+1)}{3(n-1)}\1_{n-1}' \\
\frac{(n-2)}{3(n-1)} \1_{n-1}& \frac{(n-2)^2}{9(n-1)^2}J_{n-1} &\frac{-(n-2)(n+1)}{9(n-1)^2}J_{n-1} \\
\frac{-(n+1)}{3(n-1)}\1_{n-1} & \frac{-(n-2)(n+1)}{9(n-1)^2}J_{n-1} & \frac{(n+1)^2}{9(n-1)^2}J_{n-1}
\end{array}\right].
\end{equation}
We note from (\ref{AA}) that
\begin{equation}\label{ab}
A=\frac{9(n-1)}{(n+4)^2}yy'={\lambda}\frac{yy'}{{\|y\|}^2}.
\end{equation}
Put $\delta= \frac{n-1}{2}$.
From Theorem \ref{eee},  $q_\delta =(0,u_\delta ',v_\delta')'$ where \[u_\delta =(\underbrace{1,-1,1,-1,\dotsc,1,-1}_{n-1})'\]  and   $v _\delta  =\0.$
By (\ref{h22}) \[T(r,s)=(-1)^{r+s}.\]
By a direct verification, we see that 
$u_\delta  u_\delta '=T.$
Hence
\begin{equation}\label{qkodd}
{ {q_\delta {q_\delta }}^{\ast}} =\left[ \begin{array}{cccc}
 0&\0&\0\\
\0& u_{\delta }u_{\delta  }' 
 &O\\
\0&O
&O\end{array}\right]= \left[ \begin{array}{cccc}
 0&\0&\0\\
\0& T 
 &O\\
\0&O
&O\end{array}\right].
\end{equation}
Also,
\begin{equation}\label{qkodd2}
{\langle {q_\delta },}{q_\delta\rangle  }=\langle u_\delta,u_\delta\rangle  + \langle v_\delta,v_\delta \rangle =n-1.
\end{equation}
Using (\ref{qkodd}) and (\ref{qkodd2}), we deduce
\begin{equation}\label{h}
 -{\frac{2}{\theta_{\delta  }}  {\q_{\delta   }}  {\q_{\delta  }}^{\ast}}=  \frac{ {q_{ \delta}}  {q_{  \delta}}^{\ast}}{\langle {q_{  \delta}},  {q_{\delta }\rangle}}=\frac{1}{n-1}\left[ \begin{array}{cccc}
 0&\0&\0\\
\0& T 
 &O\\
\0&O
&O\end{array}\right] =H.
\end{equation}
 Since $\rank(U)=n-1$, by
spectral decomposition, 
\begin{equation}\label{l1}
\begin{aligned}
 U^{\dagger}&={\lambda}\frac{yy'}{{\|y\|}^2}-\sum_{k=1}^{n-2} {\frac{2}{\theta_{k}}  {\q_{k}}  {\q_{k}}^{\ast}} 
 \\
 &={\lambda}\frac{yy'}{{\|y\|}^2}- {\frac{2}{\theta_{\delta}} { \q_{\delta}} {\q_{ \delta}}^{\ast}}
  -\sum_{k \in \nabla}   {\frac{2}{\theta_{k}} {\q_{k}}{\q_{k}}^{\ast}}. 
 \end{aligned}
\end{equation}
In view of (\ref{ab}) and (\ref{h}), (\ref{l1}) becomes   
\begin{equation*}
 U^{\dagger}=A+H-\sum_{k \in \nabla}  {\frac{2}{\theta_{k}} {\q_{k}}{\q_{k}}^{\ast}}.
\end{equation*}
Since ${\q_{k}}=\frac{{q_{k}}}{\|q_k\|}$, we get
\[ U^{\dagger}=A+H
- \sum_{k \in \nabla}    {(\frac{2}{\theta_{k}}) }   \frac{ {q_{k}}  {q_{k}} ^{\ast }   }{ { 
{\langle {q_{k}}},{q_{k}}\rangle }}.\]
This completes the proof.
\end{proof}
\end{lemma}

\begin{theorem}\label{lapevo}
$L_{odd}=   U^{\dagger}  $
\end{theorem}
\begin{proof}

By Lemma \ref{specU2}, we have
\begin{equation*}\label{r12}
 U^{\dagger}=A+H
- \sum_{k \in \nabla}   {(\frac{2}{\theta_{k}}) }   \frac{ {q_{k}}  {q_{k}} ^{\ast }   }{ \langle{ 
{{q_{k}}},{q_{k}}\rangle}}.
\end{equation*}
Since $U^{\dagger}$, $A$ and $H$ are real matrices,
 \begin{equation}\label{r22}
U^{\dagger}= A+H+\sum_{k \in \nabla} (-\frac{2}{\theta_{k}})\re\frac{{q_{k}}{q_{k}} ^{\ast}}{{{\langle {q_{k}}},{q_{k}}\rangle }}.
\end{equation}
Since $ \Delta_k={-(\frac{2}{\theta_{k}}} \frac{1}{\langle q_k,q_k \rangle})\re q_{k}q_{k}^{*}$,
 \begin{equation}\label{ll42}
U^{\dagger}=A+H+\sum_{k \in \nabla}  \Delta_k.
\end{equation}
By  item (v) of Lemma \ref{d1+},  
\[ \sum_{k \in \nabla} \Delta_k=\sum_{k=1}^{\frac{n-3}{2}}B_k.\]
Thus, by (\ref{ll42}), 
\[U^{\dagger}=A+H+\sum_{k=1}^{\frac{n-3}{2}}B_k. \]
In view of  (\ref{speciallapo}),
\[L_{odd} =A+H+\sum_{k=1}^{\frac{n-3}{2}}B_k.\]
Hence, \[U^{\dagger}=L_{odd}.\]
This completes the proof.
\end{proof}

\subsection{The vector $E^\dag\1$}
 \begin{lemma}\label{uvalue2}
\[E^{\dagger}\1=(\frac{-3n+13}{(n^2+3n-4)}, \frac{-n+6}{(n^2+3n-4)}\1_{n-1}',\frac{1}{n+4}\1_{n-1}' )'.\]
\end{lemma}
\begin{proof}
The proof is similar to Lemma \ref{uvalue}.
\end{proof}

\subsection{Proof of main result}
We now prove our main result now.
\begin{theorem}
Let $n \geq 4$ be an even integer. Then,  
\[E^\dag = -\frac{1}{2}L_{odd}+\frac{n-1}{2}uu',\]
where 
\[u= \huge(\frac{-(3n-13)}{(n^2+3n-4)},~ \frac{-(n-6)}{(n^2+3n-4)}\1_{n-1}',~\frac{1}{(n+4)}\1_{n-1}' \huge)'.\]
\end{theorem}

\begin{proof}Since $E$ is an EDM and $\1'E^\dag\1>0$, by the result in (K4),
\[E^{\dagger}=-\frac{1}{2}U^{\dagger}+\frac{1}{\1'E^\dag \1}(E^\dag \1)(E^\dag \1)^{'}. \]
Using Theorem  \ref{lapevo}  and  Lemma \ref{uvalue2},
we get
$L_{odd}=U^\dag$ and $u=E^\dag \1$.
In view of Lemma \ref{d1+},  we have $\1'E^\dag \1=\frac{2}{n-1}$.
Hence,
\[E^{\dagger}=-\frac{1}{2}L_{odd}+\frac{n-1}{2}uu^{'}. \]
\end{proof}

\subsection{Properties of $L_{even}$ and $L_{odd}$}
The special Laplacian matrices satisfy the following.
\begin{theorem}
The matrices $L_{even}$ and $L_{odd}$ have the following properties:
\begin{enumerate}
\item[\rm (i)] $L_{even}$ and $L_{odd}$ 
are positive semidefinite.
\item[\rm (ii)] Each row sum of $L_{even}$ and $L_{odd}$ is $0$.
\item[\rm (iii)] $\rank(L_{even})=\rank(L_{odd})=n-1$.
\end{enumerate}
\end{theorem}
\begin{proof}
By Theorem $\ref{lapev}$,
\[L_{even}=(-\frac{1}{2} PDP)^{\dag}. \]
As $D$ is a EDM, in view of (K2), $x'Dx \leq 0$ for all $x \in \1^\perp$. Since 
$P$ is the orthogonal projection onto $\1^{\perp}$,
$-PDP$ is positive semidefinite. This implies $L_{even}$ is positive semidefinite.
As $P \1=0$, $L_{even} \1=\0$. So, each row sum of
$L_{even}$ is $0$. Since $L_{even}$ has $n-1$ non-zero eigenvalues, 
$\rank(L_{even})=n-1$.

By a similar argument, we see that $L_{odd}$ satisfies (i), (ii) and (iii).
\end{proof}

R. Balaji and Vinayak Gupta \\
Department of Mathematics \\
Indian Institute of Technology -Madras \\
Chennai 600036 \\
India.
\end{document}